\documentclass[final,leqno,onefignum,onetabnum]{siamltex1213}

\usepackage{rotating}
\usepackage{url}
\usepackage{float}
\usepackage{graphicx}
\usepackage{color}
\usepackage{latexsym}
\usepackage{amsfonts}
\usepackage{amsmath}
\usepackage{psfrag}
\usepackage{mathrsfs}
\usepackage{algpseudocode}
\usepackage[ruled]{algorithm}
\usepackage[]{todonotes}

\newcommand{\LL}{\mathscr{L}}
\newcommand*{\lcdot}{\raisebox{-0.5ex}{\scalebox{1.5}{$\cdot$}}}
\def\R{\mathbb{R}}
\def\Z{\mathbb{Z}}

\def\vareps{\varepsilon}

\def\FASTQPA{\texttt{FAST-QPA}}

\title{A Feasible Active Set Method with Reoptimization for Convex
  Quadratic Mixed-Integer Programming\thanks{The second author has
    been partially supported by the German Research Foundation (DFG) under grant
    BU 2313/4.
}}

\author{Christoph~Buchheim\thanks{Fakult\"at f\"ur Mathematik, Technische Universit\"at Dortmund, Vogelpothsweg 87, 44227 Dortmund, Germany
    \email{christoph.buchheim@tu-dortmund.de}} 
  \and 
  Marianna~De~Santis\thanks{Fakult\"at f\"ur Mathematik, Technische Universit\"at Dortmund, Vogelpothsweg 87, 44227 Dortmund, Germany
    \email{marianna.de.santis@math.tu-dortmund.de}}
  \and
  Stefano~Lucidi\thanks{Dipartimento di Ingegneria Informatica
    Automatica e Gestionale, Sapienza Universit\`{a} di Roma, Via
    Ariosto, 25, 00185 Roma, Italy
    \email{stefano.lucidi@dis.uniroma1.it}}
  \and
  Francesco~Rinaldi\thanks{Dipartimento di Matematica, Universit\`a
    di Padova, Via Trieste, 63, 35121 Padova, Italy
    \email{rinaldi@math.unipd.it}}
  \and
  Long~Trieu\thanks{Fakult\"at f\"ur Mathematik, Technische Universit\"at Dortmund, Vogelpothsweg 87, 44227 Dortmund, Germany
    \email{long.trieu@math.tu-dortmund.de}}
}

\begin{document}
%\today
\maketitle
\slugger{mms}{xxxx}{xx}{x}{x--x}%slugger should be set to mms, siap, sicomp, sicon, sidma, sima, simax, sinum, siopt, sisc, or sirev

\begin{abstract}
  We propose a feasible active set method for convex quadratic programming problems with non-negativity constraints.
  This method is specifically designed to be embedded into a branch-and-bound algorithm
  for convex quadratic mixed integer programming problems.
  The branch-and-bound algorithm generalizes the approach for unconstrained 
  convex quadratic integer programming proposed by Buchheim, Caprara and 
  Lodi~\cite{buchheim2012} to the presence of linear constraints.
  The main feature of the latter approach consists in a sophisticated preprocessing
  phase, leading to a fast enumeration of the branch-and-bound nodes. 
  Moreover, the feasible active set method takes advantage of this preprocessing phase and is well suited
  for reoptimization. 
  Experimental results for randomly generated instances show that
  the new approach significantly outperforms the MIQP solver of \texttt{CPLEX 12.6} 
  for instances with a small number of constraints.
\end{abstract}



    \pagestyle{myheadings}
    \thispagestyle{plain}
    \markboth{C. BUCHHEIM, M. DE SANTIS, S. LUCIDI, F. RINALDI, AND
      L. TRIEU}{FEASIBLE ACTIVE SET METHOD FOR CONVEX MIQP}

    \section{Introduction}
    We consider mixed-integer optimization problems with strictly convex quadratic objective
    functions and linear constraints:
    \begin{equation}\tag{MIQP}\label{QIP}
      \begin{array}{l l}
        \min & f (x) = x^{\top}Qx + c^{\top}x + d\\
        \textnormal{s.t. }& Ax \leq b\\
        & x_i\in \Z,\ i=1,\dots,n_1\\
        & x_i\in \R,\ i=n_1+1,\dots,n
      \end{array}
    \end{equation}
    where $Q \in \R^{n\times n}$ is a positive definite matrix, $c \in \R^n$, $d
    \in \R$, $A \in \R^{m\times n}$, $b \in \R^m$, and~$n_1\in\{0,\dots,n\}$ is the number of integer variables.
    Up to now, all solution methods for convex mixed integer quadratic
    programming are based on either cutting planes~\cite{agrawal1974}, 
    branch-and-cut~\cite{bienstock1996}, 
    Benders decomposition~\cite{lazimy1982,lazimy1985} 
    or branch-and-bound~\cite{al-khayyal1990,fletcher1998,leyffer1993}. 
    Numerical experiments by Fletcher and Leyffer~\cite{fletcher1998} 
    showed that branch-and-bound is the most effective approach out of all the common methods, 
    since the continuous relaxations are very easy to solve. 
    Standard commercial solvers that can handle~\eqref{QIP} include 
    \texttt{CPLEX}~\cite{cplex-url}, \texttt{Xpress}~\cite{express-url}, 
    \texttt{Gurobi}~\cite{gurobi-url} and \texttt{MOSEK}~\cite{mosek-url}, 
    while \texttt{Bonmin}~\cite{bonami:2008} and \texttt{SCIP}~\cite{achterberg2009} 
    are well-known non-commercial software packages being capable of solving~\eqref{QIP}.

    Many optimization problems in real world applications can be
    formulated as convex mixed-integer optimization problems with few
    linear constraints, e.g., the classical mean-variance optimization
    problem (MVO) for the selection of portfolios of
    assets~\cite{cornuejols2006}.  Thus, the design of effective
    algorithms for this class of problems plays an important role both in
    theory and in practice.

    Our approach is based on a branch-and-bound scheme that enumerates
    nodes very quickly. Similar to Buchheim et al.~\cite{buchheim2012}, we
    fix the branching order in advance, thus losing the flexibility of
    choosing sophisticated branching strategies but gaining the advantage
    of shifting expensive computations into a preprocessing phase. In each
    node the dual problem of the continuous relaxation is solved in order
    to determine a local lower bound. Since all constraints of the
    continuous relaxation of~\eqref{QIP} are affine, strong duality holds
    if the primal problem is feasible. The dual problem of the continuous
    relaxation is again a convex Quadratic Programming (QP) problem but contains
    only non-negativity constraints, for its solution we devise a specific
    feasible active set method.
    By considering the dual problem, it suffices to find an approximate
    solution, as each dual feasible solution yields a valid lower bound.
    We can thus prune the branch-and-bound node as soon as the current
    upper bound is exceeded by the value of any feasible iterate produced
    in a solution algorithm for the dual problem. 

    Another feature of our algorithm is the use of warmstarts: after each
    branching, corresponding to the fixing of a variable, we pass the
    optimal solution from the parent node as starting point in the
    children nodes. This leads to a significant reduction in the average
    number of iterations needed to solve the dual problem to optimality.

    If the primal problem has a small number of constraints, the dual
    problem has a small dimension. Using sophisticated incremental
    computations, the overall running time per node in our approach is
    linear in the dimension $n$ if the number of constraints is fixed and
    when we assume that the solution of the dual problem, which is of
    fixed dimension in this case, can be done in constant time. In this sense, our
    approach can be seen as an extension of the algorithm devised
    in~\cite{buchheim2012}.

    The paper is organized as follows. In Section~\ref{sec:feas_as} we
    present the active set method for convex QP problems with
    non-negativity constraints. The properties of the proposed active set
    estimation are discussed and the convergence of the algorithm is
    analyzed.  Section~\ref{sec:as_in_bnb} presents an outline of the
    branch-and-bound algorithm: we discuss the advantages of considering
    the corresponding dual problem instead of the primal one. Afterwards,
    we explain the idea of reoptimization, using warmstarts within the
    branch-and-bound scheme. The end of the section deals with some tricks
    to speed up the algorithm by using incremental computations and
    an intelligent preprocessing. In Section~\ref{sec:exp} we present computational
    results and compare the performance of the proposed algorithm, applied
    to randomly generated instances, to the MIQP solver of \texttt{CPLEX 12.6}. Section~\ref{sec:summary} concludes.

    \section{A Feasible Active Set Algorithm for Quadratic Programming Problems with Non-negativity Constraints}\label{sec:feas_as}

    The vast majority of solution methods for (purely continuous)
    quadratic programs can be categorized into either interior point
    methods or active set methods (see~\cite{nocedal1999} and references
    therein for further details).  In interior point methods, a sequence
    of parameterized barrier functions is (approximately) minimized using
    Newton's method. The main computational effort consists in solving the
    Newton system to get the search direction.  In active set methods, at
    each iteration, a working set that estimates the set of active
    constraints at the solution is iteratively updated. Usually, only a
    single active constraint is added to or deleted from the active set at each
    iteration. However, when dealing with simple constraints, one can use
    projected active set methods, which can add to or delete from the
    current estimated active set more than one constraint at each
    iteration, and eventually find the active set in a finite number of
    steps if certain conditions hold. %, thus obtaining better performance.

    An advantage of active set methods is that they are well-suited for
    warmstarts, where a good estimate of the optimal active set is used to
    initialize the algorithm. This is particularly useful in applications
    where a sequence of QP problems is solved, e.g., in a sequential
    quadratic programming method.  Since in our branch-and-bound framework
    we need to solve a large number of closely related quadratic programs,
    using active set strategies seems to be a reasonable choice.

    In this section, we thus consider QP problems of the form
    \begin{equation}\tag{QP}\label{quadpr}
      \begin{array}{l l}
        \min & q(x) = x^{\top} Q x +  c^{\top} x +  d \\
        \textnormal{s.t. }& x\ge 0\\
        & x \in \R^m,
      \end{array}
    \end{equation}
    where $Q\in \R^{ m\times m}$ is positive semidefinite, $c \in \R^{m}$
     and $d\in \R$. 
%We assume in the following that the minimum
%     in~\eqref{quadpr} is finite. 
    We describe a projected active set method
    for the solution of such problems that tries to exploit the
    information calculated in a preprocessing phase at each level of the
    branch-and-bound tree. Our method is inspired by the work of
    Bertsekas~\cite{Bertsekas1982}, where a class of active set projected
    Newton methods is proposed for the solution of problems with
    non-negativity constraints.  The main difference between the two
    approaches is in the way the active variables are defined and
    updated. On the one side, the method described in~\cite{Bertsekas1982}
    uses a linesearch procedure that both updates active and non-active
    variables at each iteration.  On the other side, our method, at a
    given iteration, first sets to zero the active variables (guaranteeing
    a sufficient reduction of the objective function), and then tries to
    improve the objective function in the space of the non-active
    variables. This gives us more freedom in the choice of the stepsize
    along the search direction, since the active variables are not
    considered in the linesearch procedure.

    In the following we denote by $g(x)\in\R^m$ and $H(x)\in\R^{m\times
      m}$ the gradient vector and the Hessian matrix of the objective
    function $q(x)$ in Problem~\eqref{quadpr}, respectively.  Explicitly,
    we have
    $$g(x)= 2Qx + c, \quad H(x) = 2Q\;.$$ 
    Given a matrix $M$, we further
    denote by $\lambda_{max}(M)$ the maximum eigenvalue of~$M$.  Given a
    vector $v\in \R^m$ and an index set $I \subseteq\{1,\ldots,m\}$, we
    denote by $v_I$ the sub-vector with components $v_i$ with $i\in I$.
    Analogously, given the matrix $H\in \R^{m\times m}$ we denote by
    $H_{I\,I}$ the sub-matrix with components $h_{i,j}$ with $i,j \in
    I$. Given two vectors $v,y \in \R^m$ we denote by $\max\{v,y\}$ the
    vector with components $\max\{v_i,y_i\}$, for~$i\in \{1,\ldots,m\}$. The
    open ball with center $x$ and radius $\rho>0$ is denoted by ${\cal
      B}(x,\rho)$.  Finally, we denote the projection of a point $z\in
    \R^m$ onto $\R^m_+$ by $[z]^\sharp: = \max\{0,
    z\}$.

  \subsection{Active Set Estimate}\label{estimate}

    The idea behind active set methods in constrained nonlinear
    programming problems is that of correctly identifying the set of
    active constraints at the optimal solution $x^\star$.
    \begin{definition}\label{def:activeset}
      Let $x^\star\in \R^m$ be an optimal solution for
      Problem~(\ref{quadpr}). We define as \emph{active set} at~$x^\star$ the
      following:
      \begin{equation}\label{AS}
        {\cal \bar A}(x^\star)=\big\{i\in \{1,\ldots,m\}: x^\star_i=0\big\}.\nonumber
      \end{equation}
      We further define as \emph{non-active set} at $x^\star$ the
      complementary set of ${\cal \bar A}(x^\star)$:
      \begin{equation}\label{NAS}
        {\cal \bar N}(x^\star)=\{1,\ldots,m\}\setminus {\cal \bar A}(x^\star)=\big\{i\in \{1,\ldots,m\}:x^\star_i> 0\big\}.\nonumber
      \end{equation}
    \end{definition}

    Our aim is to find a rule that leads us to the identification of the
    optimal active set for Problem~(\ref{quadpr}) as quickly as possible.
    The rule we propose is based on the use of multiplier functions and
    follows the ideas reported in \cite{facchinei1995}.
    \begin{definition}\label{def:est}
      Let $x\in \R^m$. We define the following sets as estimates of the
      non-active and active sets at $x$:
      \begin{equation}\label{nas1}
        {\cal N}(x)=\{i\in \{1,\ldots,m\}: x_i>\varepsilon\, g_i(x) \}\nonumber
      \end{equation}
      and
      \begin{equation}\label{as1}
        {\cal A}(x)=\{1,\ldots,m\}\setminus {\cal N}(x),\nonumber
      \end{equation}
      where $\varepsilon>0$ is a positive scalar.
    \end{definition}

    The following result can be proved~\cite{facchinei1995}:
    \begin{theorem}\label{activeest}
      Let $x^\star\in \R^m$ be a solution of Problem~\eqref{quadpr}. Then, there
      exists a neighborhood of~$x^\star$ such that, for each $x$ in this neighborhood, we have
      \begin{equation}\label{stcmp}
        {\cal \bar A}^+(x^\star)\subseteq{\cal A} (x) \subseteq{\cal \bar A}(x^\star),\nonumber
      \end{equation}
      with
      ${\cal \bar A}^+(x^\star)={\cal \bar A}(x^\star)\cap \{i \in \{1,\ldots,m\}: g_i(x^\star)>0\}$.
    \end{theorem}

    Furthermore, if strict complementarity holds, we can state the following:
    \begin{corollary}\label{activeestc}
      Let $x^\star\in \R^m$ be a solution of Problem~\eqref{quadpr} where strict complementarity holds. Then, there
      exists a neighborhood of $x^\star$ such that, for each $x$ in this neighborhood, we have
      \begin{equation}\label{stcmp2}
        {\cal A} (x) ={\cal \bar A}(x^\star).\nonumber
      \end{equation}
    \end{corollary}

   \subsection{Outline of the Algorithm}\label{subsec:fastqp}
    We now give an outline of our Feasible Active SeT Quadratic
    Programming Algorithm (\FASTQPA) for solving
    Problem~\eqref{quadpr}; see Algorithm~\ref{fig:FAST-QPA}.
    \begin{algorithm}                      % enter the algorithm environment
      \caption{Feasible Active SeT Quadratic Programming Algorithm (FAST-QPA)}          % give the algorithm a caption
      \label{fig:FAST-QPA} % and a label for \ref{} commands later in the document
      \begin{algorithmic} % enter the algorithmic environment
        \par\vspace*{0.1cm}
      \item $0$\hspace*{0.5truecm} {\bf Fix}  $\delta \in (0,1)$, $\gamma\in (0,\frac 1 2)$ and $\varepsilon$ satisfying~\eqref{eps-prop2}
      \item $1$\hspace*{0.5truecm} {\bf Choose} $x^0\in \R^n$, 
      \item$2$\hspace*{0.5truecm} {\bf For } $k=0, 1,\ldots$
      \item$3$\hspace*{1.0truecm} {\bf If} $x^k$ is optimal {\bf then} STOP
        %$3$\hspace*{1.0truecm} {\bf Compute} $g^k=A^{\top} r$  \\
      \item$4$\hspace*{1.0truecm} {\bf Compute} ${\cal A}^k:={\cal A}(x^k)$ and  ${\cal N}^k:={\cal N}(x^k)$
      \item$5$\hspace*{1.0truecm} {\bf Set} $\tilde x^{k}_{{\cal A}^k} = 0$ and $\tilde x^{k}_{{\cal N}^k} = x^k_{{\cal N}^k}$
        %$6$\hspace*{1.0truecm} {\bf Update} $r=r+\sum_{i\in {\cal A}^k} A_i(y^{0,k}_i-x^k_i)$  \\
      \item$6$\hspace*{1.0truecm} {\bf Set} $d^{k}_{{\cal A}^k} = 0$
      \item$7$\hspace*{1.0truecm} {\bf Compute} a gradient related direction
        $d^{k}_{{\cal N}^k}$ in $\tilde x^k$ {\bf and} the point $\bar x^k$
      \item$8$\hspace*{1.0truecm} {\bf If} $\bar x^k$ is optimal {\bf then} STOP
        %\item$6$\hspace*{1.0truecm} {\bf Compute} a direction $d^{k}_{{\cal N}^k}$ such that
        %\begin{eqnarray}
        % {d^{k}}^{\top}_{{\cal N}^{k}} g(\tilde x^k)_{{\cal N}^k}&\leq& -\sigma_1 \|g(\tilde x^k)_{{\cal N}^k}\|^2\nonumber\\
        % \|d^{k}_{{\cal N}^k}\|&\leq& \sigma_2 \|g(\tilde x^k)_{{\cal N}^k}\|\nonumber
        %\end{eqnarray}
        %\item\hspace*{1.2truecm} with $\sigma_1,\sigma_2>0$
        %\item$7$\hspace*{1.0truecm}{\bf Choose} $\delta \in (0,1)$, $\gamma\in (0,\frac 1 2)$ 
      \item$9$\hspace*{1.0truecm} {\bf Set} $\alpha^k = \delta^j$ where $j$ is the first non-negative integer for which
        \begin{equation*}%\label{suff_dec}
          q([\tilde x^k+\delta^{j} d^k]^{\sharp})\leq q(\tilde x^k)+\gamma\, \delta^{j}\, g(\tilde x^k)^{\top} d^k
        \end{equation*}
      \item$10$\hspace*{1.0truecm} {\bf Set}
        $$x^{k+1}=[\tilde x^k+\alpha^k d^k]^{\sharp}$$
      \item $11$\hspace*{0.4truecm} {\bf End For }
        \par\vspace*{0.1cm}
      \end{algorithmic}
    \end{algorithm}

    At each iteration~$k$, the algorithm first determines the two
    sets~${\cal N}^k:={\cal N}(x^k)$ and~${\cal A}^k:={\cal A}(x^k)$
    according to Definition~\ref{def:est}. Then, the variables
    belonging to ${\cal A}^k$ are set to zero, thus obtaining a new
    point $\tilde x^k$, and a search direction $d^k$ is calculated.
    More specifically $d_{{\cal A}^k}^k$, the components of $d^k$
    related to ${\cal A}^k$, are set to zero, while a gradient related
    direction $d^{k}_{{\cal N}^k}$ is calculated in ${\cal N}^k$. Here we define 
    that a direction $d_{{\cal N}^k}$ is gradient
    related  at $\tilde x^k$ (see e.g. \cite{Ber99}) if there exist $\sigma_1, \sigma_2>0$ such that
    \begin{eqnarray}
      d_{{\cal N}^k}^{\top} g(\tilde x^k)_{{\cal N}^k}&\leq& -\sigma_1 \|g(\tilde x^k)_{{\cal N}^k}\|^2,\label{eq1}\\
      \|d_{{\cal N}^k}\|&\leq& \sigma_2 \|g(\tilde x^k)_{{\cal N}^k}\|.\label{eq2}
    \end{eqnarray}
    In order to obtain the direction $d^{k}_{{\cal N}^k}$, we solve the 
    following
    unconstrained quadratic problem in the subspace of non-active
    variables,
    \begin{equation}\tag{QP$_{k}$}\label{quadprN}
\begin{array}{l l}
\min & q^k(y) = y^\top \widetilde Q y + \tilde c^\top y\\
s.t. & y \in \R^{|{\cal N}^k|},
\end{array}
\end{equation}
where $\widetilde Q = Q_{{\cal N}^k\, {\cal N}^k}$ and $\tilde c = g(\tilde x^k)_{{\cal N}^k}$.
    In particular, we apply a conjugate gradient type algorithm (see Algorithm~\ref{fig:TN}).
    Algorithm~\ref{fig:TN} is basically a modification of the conjugate gradient method for quadratic problems that embeds 
    a truncated Newton like test for the calculation of a gradient related direction (Step 3 of Algorithm ~\ref{fig:TN}).\\ 
    \begin{algorithm} % enter the algorithm environment
    \caption{Calculation of the direction $d^{k}_{{\cal N}^k}$ and of the point $\bar x^k$} 
    \label{fig:TN} % and a label for \ref{} commands later in the document
    \begin{algorithmic} % enter the algorithmic environment
    \par\vspace*{0.1cm}
      \item $1$\hspace*{0.5truecm} {\bf Set} $y^0=0$, $\tilde g^0=\widetilde Q \,\tilde x^{k}_{{\cal N}^k}+ \tilde c$, $p^0=-\tilde g^0$, $\eta >0$, $l=0$ and $check =$ true;
      \par\vspace*{0.1cm}
      \item$2$\hspace*{0.5truecm} {\bf While} $p{^l}^\top \widetilde Q p{^l} > 0$
      \par\vspace*{0.1cm}
      \item$3$\hspace*{1.0truecm} {\bf If} $ p{^l}^\top \widetilde Q p{^l} \le \eta \|p^l\|^2 $ {\bf and} $check = \mbox{true}$ {\bf then}
      $$
      \hat y = \left\{
      \begin{array}{l l}
	-\tilde g^0, &\mbox{if }\, l=0\\
       y^l,&\mbox{if }\, l>0
      \end{array}
      \right.
      ,\quad  check =\mbox{false};
      $$
      
      \item$4$\hspace*{1.0truecm} {\bf End If}
      %$3$\hspace*{1.0truecm} {\bf Compute} $g^k=A^T r$ ;  \\
      \item$5$\hspace*{1.0truecm} {\bf Compute} stepsize along the search direction $p^l$;
     \begin{eqnarray}
	\alpha^l&=&\frac{{\|\tilde g^l\|}^2}{{p^l}^\top\widetilde Q p^l}\nonumber
     \end{eqnarray}
      \item$6$\hspace*{1.0truecm} {\bf Update} point $y^{l+1}=y^l+\alpha^l\, p^l$;
      \item$7$\hspace*{1.0truecm} {\bf Update} gradient of the quadratic function $\tilde g^{l+1}=\tilde g^l+\alpha^l\, \widetilde Q\, p^l$;
      \item$8$\hspace*{1.0truecm} {\bf Compute} coefficient
     \begin{eqnarray}
	\beta^{l+1}&=&\frac{\|\tilde g^{l+1}\|^2}{\|\tilde g^l\|^2};\nonumber
     \end{eqnarray}
      \item$9$\hspace*{1.0truecm} {\bf Determine} new conjugate direction $p^{l+1}=- \tilde g^{l+1}+\beta^{l+1} p^l$;
      \item$10$\hspace*{0.85truecm} {\bf Set} $l=l+1$;
      \item $11$\hspace*{0.4truecm} {\bf End While}
      \item$12$\hspace*{0.4truecm} {\bf If}  $check = \mbox{true}$ {\bf then}
      $$
      \hat y = \left\{
      \begin{array}{l l}
	-\tilde g^0, &\mbox{if }\, l=0\\
       y^{l},&\mbox{if }\, l>0
      \end{array}
      \right.
      $$
      
      \item$13$\hspace*{0.4truecm} {\bf End If}
      
      \par\vspace*{0.1cm}
      \item $14$\hspace*{0.4truecm} {\bf Return} $ d^{k}_{{\cal N}^k}=\hat y$ and $\bar x^{k}$,
      where 
      $$
      \bar x^k_{{\cal A}^k} = 0,\quad\bar x^k_{{\cal N}^k} =  \left\{
      \begin{array}{l l}
	-\tilde g^0, &\mbox{if }\, l=0\\
       y^{l},&\mbox{if }\, l>0
      \end{array}
      \right..$$
      \par\vspace*{0.1cm}
    \end{algorithmic}
    \end{algorithm}
    %\par\smallskip\noindent
%%%%%%%%%%%%%%%%%%%% direction gradient related %%%%%%%%%%%%%%%%%
%Now, we state two theoretical results, proved in~\cite{pshenichny1978}, that will help us to better understand 
%the properties of Algorithm \ref{fig:TN}:
We report two theoretical results, proved in~\cite{pshenichny1978}, that help us to understand 
the properties of Algorithm \ref{fig:TN}:
\begin{proposition}\label{prop:CGstop}
 Let Problem~\eqref{quadprN} admit an optimal solution and matrix $\widetilde Q$ be positive semidefinite. Then,
 Algorithm~\ref{fig:TN} terminates with an optimal solution of Problem~\eqref{quadprN} in less than $|{\cal N}^k|$ iterations.
\end{proposition} 
\begin{proposition}\label{prop:CGstop2}
Let matrix $\widetilde Q$ in Problem~\eqref{quadprN} be positive semidefinite.
Then at least one direction $p^l$, $1\leq l\leq |{\cal N}^k|-1$, computed by Algorithm~\ref{fig:TN} satisfies the condition 
$$
{p^l}^\top\widetilde Q{p^l}=0.
$$
\end{proposition} 
By Propositions~\ref{prop:CGstop} and~\ref{prop:CGstop2}, we are able to prove the following:
\begin{proposition}\label{TNprop}
Let matrix $\widetilde Q$ in Problem~\eqref{quadprN} be positive semidefinite. Algorithm~\ref{fig:TN} terminates after a finite number of iterations, 
returning a gradient related direction $ d^{k}_{{\cal N}^k}$. 
Furthermore, if Problem \eqref{quadprN} admits an optimal solution, then $y^l$ is an optimal solution of Problem \eqref{quadprN}.
\end{proposition}
\begin{proof} By the results reported in \cite{dembo1982, grippo1989} we have that  
direction $ d^{k}_{{\cal N}^k}$ satisfies~\eqref{eq1} and~\eqref{eq2} and hence is
gradient related at $\tilde x^k$. Furthermore, by taking into account Propositions~\ref{prop:CGstop} and~\ref{prop:CGstop2}, the 
rest of the result follows.
\end{proof}

  %   Algorithm~\ref{fig:TN} always stops in a finite number of iterations giving a direction $d^{k}_{{\cal N}^k}$ that satisfies~\eqref{eq1} and~\eqref{eq2} and hence is
  %   gradient related. When the minimum of Problem~\eqref{quadprN} is finite, we further have that Algorithm~\ref{fig:TN} produces an optimal solution $y^{l-1}$ of
  %   Problem~\eqref{quadprN} in at most $|{\cal N}^k|-1$ iterations.
  % %    \begin{remark}
According to Proposition \ref{TNprop}, when the minimum of Problem~\eqref{quadprN} is finite,  Algorithm~\ref{fig:TN} produces a gradient related direction $ d^{k}_{{\cal N}^k}$
and an optimal solution $y^{l}$ of Problem~\eqref{quadprN} (in at most $|{\cal N}^k|-1$ iterations). The point $y^{l}$ is then used for building up the candidate optimal point $\bar x^k$.
In case Problem~\eqref{quadprN} is unbounded from below, the algorithm still stops after a finite number of iterations giving a gradient related direction $ d^{k}_{{\cal N}^k}$, but the point used for generating 
$\bar x^k$ is just a combination of conjugate gradient directions generated along the iterations of the method. As we will see in Section \ref{algorithm}, calculating the point $\bar x^k$ is needed to 
guarantee, under some specific assumptions, finite convergence of Algorithm \ref{fig:FAST-QPA} (see Proposition \ref{finiteconv}).

    Note that, even if the matrix~$\widetilde Q$ is ill-conditioned, Algorithm~\ref{fig:TN} still generates a gradient related direction.
    In practice, when dealing with ill-conditioned problems, suitable preconditioning techniques can be applied
    to speed up Algorithm~\ref{fig:TN}.
  %  \end{remark}
  
    Once the direction~$d^k$ and the point $\bar x^k$ are computed, Algorithm~\ref{fig:FAST-QPA} 
    checks optimality of point~$\bar x^k$.  If $\bar x^k$ is not optimal, the algorithm generates a new point $x^{k+1}$ by means of a projected Armijo linesearch along~$d^k$. 
 		
    We finally notice that at each iteration of Algorithm~\ref{fig:FAST-QPA}, two different optimality checks are performed:
    the first one, at Step~3, to test optimality of the current iterate~$x^k$; 
    the second one, at Step~8,  to test optimality of the candidate solution~$\bar x^k$.

    \subsection{Convergence Analysis}\label{algorithm}

    The convergence analysis of \FASTQPA\ is based on two key
    results, namely Proposition~\ref{prop1} and
    Proposition~\ref{lemmastaz} stated below. These results show that the algorithm obtains a significant reduction of the objective function both when fixing to zero the variables in the active set estimate 
    and when we perform the projected Armijo linesearch.

    Proposition~\ref{prop1} completes the properties of the active set identification strategy defined in Section~\ref{estimate}.
    More specifically, it shows that, for a suitably chosen value of the parameter $\varepsilon$ appearing in Definition~\ref{def:est},
    a decrease of the objective function is achieved by simply fixing to zero
    those variables whose indices belong to the estimated active set. 
    \begin{proposition}\label{prop1}
      Assume that the parameter $\vareps$ appearing in
      Definition~\ref{def:est} satisfies
      \begin{equation}\label{eps-prop2} 0 < \varepsilon <\frac{1}{2\lambda_{max}(Q)}\;.\end{equation}
      Given the point $z$ and the sets ${\cal A}(z)$ and ${\cal N}(z)$, let $y$ be the point such that
      $$
          y_{{\cal A}(z)} = 0,\quad y_{{\cal N}(z)} = z_{{\cal N}(z)}.
      $$
      Then,
      $$
      q(y)-q(z)\leq - \frac{1}{2\vareps}\|y-z\|^2.
      $$
    \end{proposition}
    \begin{proof} 
      Define ${\cal A}={\cal A}(z)$ and ${\cal N}={\cal N}(z)$.
      By taking into account the definition of these two sets and the points $y$ and $z$, we have
      %it is possible to write:
      \begin{eqnarray*}
        && q(y) = q(z)+ g_{{\cal A}}(z)^{\top}(y-z)_{{\cal
            A}}+
        \frac{1}{2}(y-z)_{{\cal A}}^{\top} H_{ {\cal A}  {\cal A}} (y-z)_{ {\cal A}}\;.
      \end{eqnarray*}
      %where $H_{ {\cal A}  {\cal A}}$ is the Hessian matrix of the quadratic part restricted
      %to the variables with indices belonging in ${\cal A}$.
      Since $H=2Q$, the following inequality holds
      \begin{eqnarray}\nonumber
        &&q(y)
        \leq q(z)+ g_{ {\cal A}}(z)^{\top}(y-z)_{ {\cal A}}+\lambda_{max}( Q)\,\|(y-z)_{ {\cal A}}\|^2.
      \end{eqnarray}
      Using (\ref{eps-prop2}) we obtain
      \begin{eqnarray}
        &&q(y)\leq q(z)+ g_{ {\cal A}}(z)^{\top}(y-z)_{ {\cal A}}+\frac{1}{2\vareps}\|(y-z)_{ {\cal A}}\|^2\nonumber
      \end{eqnarray}
      and hence
      $$q(y)\leq q(z)+ \left(g_{{\cal A}}(z)+ \frac{1}{\vareps}(y-z)_{{\cal A}}\right)^{\top}(y-z)_{ {\cal A}}-\frac{1}{2\vareps}\|(y-z)_{ {\cal A}}\|^2. $$
      It thus remains to show
      \begin{equation*}
        \left(g_{{\cal A}}(z)+ \frac{1}{\vareps}(y-z)_{{\cal A}}\right)^{\top}(y-z)_{{\cal A}} \leq 0\;,
      \end{equation*}
      which follows from the fact that, for all $i\in  {\cal A}$,
      \begin{equation*}
        \left(g_i(z)+ \frac{1}{\vareps}(y_i-z_i)\right)(y_i-z_i) \leq 0.
      \end{equation*}
      Indeed, for all $i \in {\cal A}$ we have $z_i \ge 0$ and $y_i =
      0$, hence~$z_i-y_i=z_i\le
      \vareps\, g_i(z)$, so that
      \begin{equation*}
        g_i(z) +\frac{1}{\vareps}\,(y_i - z_i)\geq 0\;.
      \end{equation*}
    \end{proof}

    The following result shows that the projected Armijo linesearch
    performed at Step~9 of Algorithm~\ref{fig:FAST-QPA} terminates in
    a finite number of steps, and that the new point obtained
    guarantees a decrease of the objective function of~\eqref{quadpr}.
    Its proof is similar
    %follows the same arguments used in 
    to the proof of Proposition~2 in \cite{Bertsekas1982}.
    \begin{proposition}\label{lemmastaz}
      Let $\gamma\in (0,\frac 12)$. Then, for every $\bar x \in \R^n_+$ which is not optimal for Problem \eqref{quadpr}, there exist $\rho>0$ and
      $\bar{\alpha}>0$ such that
        \begin{equation}\label{LemmaSigma}
          q([\tilde x+\alpha d]^{\sharp})-q(\tilde x)\le \gamma\, \alpha\,
          d^{\top}_{{\cal N}(x)}
          g(\tilde x)_{{\cal N}(x)}
        \end{equation}
        for all $x, \tilde x \in\R^n_+$ with $x, \tilde x \in {\cal B}(\bar
        x, \rho)$ and for all
        $\alpha\in (0,\bar{\alpha}]$, where $d\in \R^n$ is the
	direction used at~$\tilde x$, and such that 
	\begin{itemize}
	 \item[(i)] $d_{{\cal A}(x)} = 0$,
	 \item[(ii)] $d_{{\cal N}(x)}$ satisfies \eqref{eq1} and \eqref{eq2}. 
	\end{itemize}
    \end{proposition}
    \noindent 
    Using Proposition~\ref{prop1} and
    Proposition~\ref{lemmastaz}, we can show that the sequence $\{q(x^k)\}$ is monotonically decreasing.
     \begin{proposition}\label{fconv}
      Let $\{x^k\}$ be the sequence produced by \FASTQPA. 
      Then, the sequence $\{q(x^k)\}$ is such that
      $$ q(x^{k+1}) \leq q(x^k)- \frac{1}{2\vareps}\|\tilde x^{k}-x^k\|^2-\sigma_1 \|g(\tilde x^k)_{{\cal N}^k}\|^2.$$
    \end{proposition}
    \begin{proof}
      Let  $\tilde x^k$ be the point produced at Step~5 of Algorithm~\ref{fig:FAST-QPA}.
      By setting $y=\tilde x^{k}$ and $z=x^k$ in Proposition~\ref{prop1}, we have
      \begin{eqnarray*}%\label{app_prop1}
        &&q(\tilde x^{k})\leq q(x^k) - \frac{1}{2\vareps}\|\tilde x^{k}-x^k\|^2\;.
      \end{eqnarray*}
      Furthermore, by the fact that we use an Armijo linesearch at Step~9 of
      Algorithm~\ref{fig:FAST-QPA} and by Proposition \ref{lemmastaz},
      we have that the chosen point $x^{k+1}$ satisfies inequality \eqref{LemmaSigma}, that is
      $$
      q(x^{k+1})-q(\tilde x^k)\le \gamma\, \alpha^k \, d^{\top}_{{\cal
          N}^k}g(\tilde x^k)_{{\cal N}^k}\;.
      $$
      By taking into account \eqref{eq1}, we thus have
       $$
      q(x^{k+1})-q(\tilde x^k)\le \gamma\, \alpha^k \, d^{\top}_{{\cal
          N}^k}g(\tilde x^k)_{{\cal N}^k} \le  -\sigma_1 \|g(\tilde x^k)_{{\cal N}^k}\|^2\;.
      $$
      In summary, we obtain
      \begin{eqnarray*}
        q(x^{k+1})&\leq& q(\tilde x^k)-\sigma_1 \|g(\tilde x^k)_{{\cal N}^k}\|^2\leq \\
      &\leq&  q(x^k)- \frac{1}{2\vareps}\|\tilde x^{k}-x^k\|^2-\sigma_1 \|g(\tilde x^k)_{{\cal N}^k}\|^2\;.\end{eqnarray*}  
    \end{proof}

    Proposition~\ref{fconv} will allow us to prove two important
    results: if the minimum of Problem~\eqref{quadpr} is finite, 
    then the stopping condition of Algorithm~\ref{fig:FAST-QPA} is met in a
    finite number of iterations; see Proposition~\ref{projgrad0}.
Otherwise, the sequence $\{q(x^k)\}$ is unbounded from below; see Proposition~\ref{qxk_inf}.
     \begin{proposition}\label{projgrad0}
     Let $\{x^k\}$ be the sequence produced by \FASTQPA. 
If the minimum of Problem~\eqref{quadpr} is finite, then 
\begin{equation*}
 \lim_{k\rightarrow\infty}\|\max\{0,- g(x^k) \}\|= 0.
\end{equation*}
\end{proposition}
\begin{proof}
By Proposition~\ref{fconv}, the sequence $\{q(x^k)\}$ is monotonically decreasing.
Since it is bounded by the minimum of~\eqref{quadpr}, we have that $\{q(x^k)\}$ converges.
In particular, this implies that both $\|\tilde x^k - x^k \|^2$ and $\|g(\tilde x^k)_{{\cal N}^k}\|^2$ 
go to zero when $k\rightarrow \infty$.
By noting that 
$$0\le\|g(\tilde x^k) - g(x^k)\| = \|Q(\tilde x^k - x^k) \| \le \|Q\|\cdot \|\tilde x^k - x^k \|\rightarrow 0$$
we obtain $\|g(x^k)_{{\cal N}^k}\|^2\rightarrow 0$ as well, and thus $\| \max\{0,-g(x^k)\}_{{\cal N}^k}\|^2 \rightarrow 0$.
On the other hand, if $i\in {{\cal A}^k}$, we have $g_i(x^k)\ge 0$, so
that % $\max\{0, -g_i(x^k)\}=0$, so that
$\|\max\{0,- g(x^k) \}_{{\cal A}^k}\|^2  =0$. This shows the result.
\end{proof}

\begin{proposition}\label{qxk_inf}
Let $\{x^k\}$ be the sequence produced by \FASTQPA. 
If Problem~\eqref{quadpr} is unbounded, then
\begin{equation*}\label{qxk}
 \lim_{k\rightarrow\infty} q(x^k)= -\infty.
\end{equation*}

\end{proposition}
\begin{proof}
By Proposition~\ref{fconv}, the sequence $\{q(x^k)\}$ is monotically decreasing.
Assume by contradiction that it is bounded and hence
$\lim_{k\rightarrow\infty} q(x^k) = \bar q$
for some $\bar q\in\R$.
Then, as in the proof of Proposition~\ref{projgrad0}, $\|g(x^k)_{{\cal N}^k}\|^2 \rightarrow 0$.
Since Problem~\eqref{quadpr} is unbounded, there must exist some $d\ge 0$
%, $d\neq 0$, 
such that $Qd =0$ and $c^\top d < 0$.
In particular,  we obtain $g(x)^\top d = c^\top d = \eta < 0$ for all $x\in \R^n$.

Therefore, since $g_i(x^k)\ge 0$ for all $i \in {{\cal A}^k}$, we have
\begin{eqnarray*}
 0>\eta = g(x^k)^\top d &=& \sum_{i\in {{\cal A}^k} } g_i(x^k)d_i + \sum_{i\in {{\cal N}^k}} g_i(x^k)d_i \\
 &\ge& \sum_{i\in {{\cal N}^k}} g_i(x^k)d_i = -\left|g(x^k)_{{\cal N}^k}^\top d_{{\cal N}^k}\right| \\
 &\ge& -\|g(x^k)_{{\cal N}^k}\| \cdot \|d_{{\cal N}^k}\|\ge -\|g(x^k)_{{\cal N}^k}\| \cdot \|d\|\;,
\end{eqnarray*}
and we get a contradiction since $\|g(x^k)_{{\cal N}^k}\|^2$ goes to
zero.
\end{proof}

    Finally, we are able to prove the main result concerning the global convergence of \FASTQPA.
    \begin{theorem}\label{teorema1}
      Assume that the minimum of Problem~\eqref{quadpr} is finite and that the  parameter $\vareps$ appearing in
      Definition~\ref{def:est} satisfies~\eqref{eps-prop2}.
      Let $\{x^k\}$ be the sequence produced by Algorithm  \FASTQPA.
      Then either an integer $\bar k\geq 0$
      exists such that $ x ^{\bar k}$ is an optimal solution for Problem~\eqref{quadpr},
      or the sequence $\{x^k\}$ is infinite
      and every limit point $x^\star$
      of the sequence is an optimal point for Problem~\eqref{quadpr}.
    \end{theorem}
    \begin{proof}
      Let $x^\star$  be any limit point of the sequence $\{x^k\} $ and let $\{x^k\}_K $ be the subsequence with
      \begin{equation*}
        \lim_{k\to\infty,\, k\in K}x^{k}=x^\star.
      \end{equation*}
      By passing to an appropriate subsequence, we may assume
      that subsets~$\bar{\cal A}\subseteq\{1,\dots,m\}$ and $\bar{\cal
        N}=\{1,\dots,m\}\setminus\bar{\cal A}$ exist such that~${\cal
        A}^k=\bar{\cal A}$ and~${\cal N}^k=\bar{\cal N}$
      for all~$k\in K$, since the number of possible choices of ${\cal A}^k$
      and ${\cal N}^k$ is finite. In order to prove that~$x^\star$ is
      optimal for Problem~\eqref{quadpr}, it then suffices to show
      \begin{itemize}
      \item[(i)] $\min \left \{ g_i (x^\star), x^\star_i \right \}=0$ for all $i\in \bar{\cal A}$, and
      \item[(ii)] $g_i(x^\star)=0$ for all $i\in \bar{\cal N}$.
      \end{itemize}
      In order to show~(i), let~$\hat \imath \in \bar{\cal A}$ and define a function $\Phi_{\hat\imath}:\R^m \rightarrow \R$ by
      \begin{equation}\nonumber
        \Phi_{\hat\imath}(x)=\min \left \{ g_{\hat\imath} (x), x_{\hat\imath} \right \}\;,
      \end{equation}
      we thus have to show $\Phi_{\hat\imath}(x^\star)=0$.
      For $k\in K$, define $\tilde y^{k}\in\R^m$ as follows:
      \begin{equation*}
        \tilde y_i^{k}=\left \{
        \begin{array}{ll}
          0 &\mbox{if} \ i=\hat\imath  \\
          x_i^k &\mbox{otherwise}.\\
        \end{array}
        \right .
      \end{equation*}
      Recalling that $\tilde x^k_{\bar{\cal A}}=0$, as set in Step~5 in
      Algorithm~\ref{fig:FAST-QPA}, and using $\hat\imath\in\bar{\cal A}$, we have
      $$ \|\tilde y^{k}-x^k\|^2=(\tilde y^{k}-x^k)_{\hat\imath}^2=( \tilde x^k-x^k)_{\hat\imath}^2\le \| \tilde x^k-x^k\|^2\;.$$
      From Proposition~\ref{fconv} and since the sequence $\{q(x^k)\}$ is bounded by the minimum of~\eqref{quadpr}, 
      we have that $\{q(x^k)\}$ converges.
      In particular, this implies that $\|\tilde x^k - x^k \|^2$ and hence  
% 
%       \begin{equation*}
%         \lim_{k \to \infty,\, k \in K} \| \tilde x^k-x^k\|^2=0\;,
%       \end{equation*}
%       %By using twice Proposition~\ref{prop1},  the first time by setting $y=\tilde y^{j,k}$ and  $z=x^k$ and the second one by setting $y= y^{0,k}$ and  $z=\tilde y^{j,k}$, we write:
%       %\begin{eqnarray}\nonumber %\label{funzione-5}
%       %&& q(\tilde y^{j,k})\leq q(x^k) - \frac{1}{2\vareps}\|\tilde y^{j,k}-x^k\|^2,\\
%       %&&\nonumber % \label{funzione-6}
%       %q(y^{0,k}) \leq q(\tilde y^{j,k}) - \frac{1}{2\vareps}\|y^{0,k}-\tilde y^{j,k}\|^2.
%       %\end{eqnarray}
%       %The previous inequalities ensure that, for all $j\in {\cal A}^k$,
%       %$$
%       %q(x^{k+1})\leq q(y^{0,k})\leq q(\tilde y^{j,k})\leq q(x^k) - \frac{1}{2\vareps}\|\tilde y^{j,k}-x^k\|^2
%       %$$
%       hence
      \begin{eqnarray}\label{convytilde}
        \lim_{k \to \infty,\, k \in K} \tilde y^{k} = x^\star\;.
      \end{eqnarray}
      By Definition~\ref{def:est}, we have $0\le x_{\hat \imath}^k\leq \vareps\, g_{\hat \imath}(x^k)$
      for all $k\in K$.
      Using Assumption~(\ref{eps-prop2}),
      there exists $\xi \ge 0$ such that
      $$\varepsilon \le  \frac{1}{2Q_{\hat \imath \hat \imath}+\xi}\;.$$
      As $\tilde y^{k}_{\hat \imath}=0$, we obtain
      \begin{equation*}
        x_{\hat \imath}^k-\tilde y^{k}_{\hat \imath}=x_{\hat \imath}^k
        \leq \vareps\, g_{\hat \imath}(x^k)\leq \frac{1}{2Q_{\hat \imath \hat \imath}+\xi} g_{\hat \imath}(x^k)
      \end{equation*}
      and hence
      \begin{equation}\nonumber
        (2Q_{\hat \imath \hat \imath}+\xi) (x_{\hat \imath}^k-\tilde y^{k}_{\hat \imath})\leq g_{\hat \imath}(x^k) ,
      \end{equation}
      which can be rewritten as follows
      \begin{equation}\nonumber
        g_{\hat \imath}(x^k)+ 2Q_{\hat \imath \hat \imath} (\tilde y^{k}_{\hat \imath}-x_{\hat \imath}^k) \geq \xi (x_{\hat \imath}^k-\tilde y^{k}_{\hat \imath})\geq 0,\\
      \end{equation}
      yielding $g_{\hat \imath}(\tilde y^{k})\geq 0$. Together with $\tilde y^{k}_{\hat \imath}=0$, we obtain
      $\Phi_{\hat\imath}(\tilde y^{k})=0$. 
      By \eqref{convytilde} and the continuity of $\Phi_{\hat\imath}$, we derive
      $\Phi_{\hat\imath}(x^\star)=0$, which
      proves~(i).
      %Thus we get a contradiction with \eqref{cont0}.

      To show~(ii), assume on contrary that $g(x^\star)_{\bar{\cal N}}\neq 0$.  
      By Proposition~\ref{lemmastaz}, there exists $\bar\alpha>0$
      such that for $k\in K$ sufficiently large
      \begin{equation*}
        q([\tilde x^k+\alpha d^k]^{\sharp})-q(\tilde x^{k})\le \gamma
        \alpha\sum_{i\in \bar{\cal N}} g_i(\tilde x^k)d^k_i
      \end{equation*}
      for all $\alpha\in (0,\bar\alpha]$, as $\tilde x^k_{\bar{\cal N}}$ converges to $x^\star_{\bar{\cal N}}$.
      As we use an Armijo type rule in Step~9 of
      Algorithm~\ref{fig:FAST-QPA}, we thus have $\delta^{j-1}\ge
      \bar\alpha$ and hence
      \begin{equation}\label{alpha_bounded}
      \alpha^k=\delta^j\ge\delta\bar\alpha\;.
      \end{equation}
      Again by Step~9, using \eqref{eq1} and \eqref{alpha_bounded}, we obtain
      \begin{eqnarray*}
        q(\tilde x^k)-q([\tilde x^k+\alpha^k d^k]^{\sharp}) & \ge & -\gamma
        \alpha^k\sum_{i\in \bar{\cal N}} g_i(\tilde x^k)d^k_i  \\
        & \ge & \gamma \sigma_1 \alpha^k\|g(\tilde x^k)_{\bar{\cal N}}\|^2\\[1ex]
        & \ge & \gamma \sigma_1 \delta\bar\alpha\|g(\tilde x^k)_{\bar{\cal
            N}}\|^2\;.
      \end{eqnarray*}
      Since $\{q(x^k)\}$ converges,
      we have that the left hand side expression converges to zero, 
      while the right hand side expression converges to
      $$\gamma \sigma_1 \delta\bar\alpha\|g(x^\star)_{\bar{\cal
            N}}\|^2>0\;,$$
        and we have the desired contradiction.
    \end{proof}

    \begin{corollary}
      If~\eqref{eps-prop2} holds, then the sequence $q(x^k)$ converges to the optimal value of~\eqref{quadpr}.
    \end{corollary}

    %%%%%%%%%%%%%%% FINITE CONVERGENCE %%%%%%%%%%%%%%
    As a final result, we prove that, under a specific assumption, %if strict complementarity holds,
    \FASTQPA\ finds an optimal solution in a finite number of iterations. 
    \begin{theorem}\label{finiteconv}
    Assume that there exists an accumulation point $x^\star$ of the sequence $\{x^k\}$ generated by \FASTQPA\ such that
    \begin{equation*}
      Q_{{\cal \hat N}\,{\cal \hat N}}\succ 0\;,
    \end{equation*}
    where ${\cal \hat N}={\cal \bar N}(x^\star)\cup \{i\in\{1,\ldots,n\}: x^\star_i = 0, \, g_i(x^\star)=0\}$.
   Then \FASTQPA\ produces a minimizer of Problem~\eqref{quadpr} in a finite number of steps.
    \end{theorem}
    \begin{proof}
       Let $\{x^k\}$ be the sequence generated by \FASTQPA\ and let $\{x^k\}_K $ be the subsequence with
      \begin{equation*}
        \lim_{k\to\infty,\, k\in K}x^{k}=x^\star.
      \end{equation*}
       By Theorem~\ref{teorema1}, the limit $x^\star$ of $\{x^k\}_K$ is a minimizer of Problem~\eqref{quadpr}.
     The active set estimation yields
     $\bar{\cal A}^+(x^\star) \subseteq {\cal A}^k \subseteq \bar{\cal A}(x^\star)$
     for sufficiently large $k$.
     Consequently 
     \begin{equation}\label{athesame2}
       \bar{\cal N}(x^\star) \subseteq {\cal N}^k\;.
     \end{equation}
     Moreover, taking into account the definition of $\hat{\cal N}$, we get
     ${\cal N}^k \subseteq \hat{\cal N}$ and thus
   \begin{equation}\label{QNk}
    Q_{{\cal N}^k\,{\cal N}^k}\succ 0.
   \end{equation}
Hence by~\eqref{athesame2},~\eqref{QNk}, 
and the fact that $g_i(x^\star) =0$ for all $i\in {\cal N}^k$ we have that solving~\eqref{quadpr} is equivalent to solving the
following problem
\begin{equation}\label{quadpropt}
      \begin{array}{l l}
        \min &x^{\top} Q x +  c^{\top} x +  d \\
        \textnormal{s.t. }& x_{{\cal A}^k} = 0,\\
        & x \in \R^m.
      \end{array}
    \end{equation}
     Since $\bar x^k_{{\cal N}^k}$ is the only optimal solution for problem~\eqref{quadprN} and $\bar x^k_{{\cal A}^k} = 0$, we conclude that~$\bar x^k$ is the optimal solution of~\eqref{quadpropt} and of~\eqref{quadpr}.
    \end{proof}
		
By Theorem~\ref{finiteconv}, we just need that the submatrix $Q_{{\cal \hat N}\,{\cal \hat N}}$
is positive definite in order to ensure finite termination of \FASTQPA. This is a weaker assumption than the positive definiteness of the full matrix~$Q$, which is usually needed to guarantee finite convergence in an active set framework (see e.g. \cite{Bertsekas1982, nocedal1999} and references therein).

Furthermore, by taking into account the scheme of our algorithm, and by
carefully analyzing the proof given above, we notice two important facts:
\begin{itemize}
\item[--] the algorithm, at iteration $k$, moves towards the optimal point $x^\star$ 
by only using an approximate solution of the unconstrained problem \eqref{quadprN}, i.e. the solution needed in the calculation of the gradient related direction;
\item[--]  once the point $x^k$ gets close enough to $x^\star$, thanks to the properties of our estimate, we can guarantee that the point $\bar x^k$ is the optimum of the original problem \eqref{quadpr}.
\end{itemize}
This explains why, in the algorithmic scheme, we calculate $\bar x^k$ and just use it in the optimality check at Step 8. 

    \section{Embedding \FASTQPA\ into a Branch-and-Bound Scheme}\label{sec:as_in_bnb}

    As shown in~\cite{buchheim2014}, the approach of embedding
    taylored active set methods into a branch-and-bound scheme is very
    promising for solving convex quadratic integer programming
    problems of type~\eqref{QIP}. In this section we shortly summarize
    the key ingredients of our branch-and-bound algorithm
    \texttt{FAST-QPA-BB} that makes use of \FASTQPA, presented in
    Section~\ref{sec:feas_as}, for the computation of lower
    bounds. The branch-and-bound algorithm we consider is based on the
    work in~\cite{buchheim2012}, where the unconstrained case is
    addressed.  As our branch-and-bound scheme aims at a fast
    enumeration of the nodes, we focus on bounds that can be computed
    quickly. A straightforward choice for determining lower bounds is
    to solve the continuous relaxation of~\eqref{QIP}.  Instead of
    considering the primal formulation of the continuous relaxation
    of~\eqref{QIP}, we deal with the dual one, which again is a convex
    QP, but with only non-negativity constraints, so that it can be
    solved by \FASTQPA.  The solution can be used as a lower bound
    for~$f$ over all feasible integer points and is as strong as the
    lower bound obtained by solving the primal problem, since strong
    duality holds.   

    Our branching strategy and its advantages are discussed in 
    Section~\ref{subsec:branching}. In Section~\ref{subsec:dual}, we have 
    a closer look at the relation between the primal and the dual problem, 
    while in Section~\ref{subsec:reopt} we shortly discuss the advantage 
    of reoptimization.  Using a predetermined branching order, some of
    the expensive calculations can be moved into a preprocessing
    phase, as described in Section~\ref{subsec:incremental}.

    \subsection{Branching}\label{subsec:branching}
    
    At every node in our branch-and-bound scheme, we branch by fixing a 
    single \textit{primal} variable in increasing distance to its value in 
    the solution of the continuous relaxation $x^\star$.
    For example, if the closest 
    integer value to $x^\star_i$ is $\lfloor  x^\star_i\rfloor$, we fix $x_i$ 
    to integer values $\lfloor x^\star_i\rfloor,\lceil x^\star_i\rceil, \lfloor  x^\star_i\rfloor-1, \lceil  x^\star_i\rceil+1$, and so on.
    %% ADD
    After each branching step, the resulting subproblem
    is a quadratic programming problem of type~\eqref{QIP} again, with a dimension
    decreased by one. We are not imposing bound constraints on the integer variables (i.e. $x_i \le \lfloor x^\star_i\rfloor$ 
    and $x_i \ge \lfloor x^\star_i\rfloor$)
    since they are taken into account as fixings (i.e. $x_i = \lfloor  x^\star_i\rfloor$) in the construction of the reduced subproblem 
    by adapting properly the matrices $Q_\ell$ and $A_\ell$, the linear term $\bar c$, the constant term $\bar d$
    and the right hand side $\bar b$ (see Section~\ref{subsec:incremental}).
    %%% 
    The branching order of these variables at every level $\ell$ is set to $x_1,\dots,x_{n-\ell}$, assuming that $\ell$ variables are 
    already fixed. Hence, at every level we have a predetermined branching order. Let $x^\star_i$ be the value of the next branching variable 
    in the continuous minimizer. Then, by the strict convexity of $f$, all consecutive lower bounds obtained by 
    fixing $x_i$ to integer values in increasing distance to $x^\star_i$, on each side of $
    x^\star_i$, are increasing. Thus, we can cut off the current node of the tree 
    and all its outer siblings as soon as we fix a variable to some value for 
    which the resulting lower bound exceeds the current best known upper bound.     
    Since $f$ is strictly convex we get a finite algorithm even
    without bounds on the variables.

    Once all integer variables have been fixed, we compute the optimal
    solution of the QP problem in the reduced continuous subspace.  If
    the computed point is feasible, it yields a valid upper bound for
    the original problem.  As our enumeration is very fast and we use
    a depth-first approach, we do not need any initial feasible point
    nor do we apply primal heuristics.
    \subsection{Dual Approach}\label{subsec:dual}

    In the following, we derive the dual problem of the continuous
    relaxation of (\ref{QIP}) and point out some advantages when using
    the dual approach in the branch-and-bound framework.  The dual can
    be computed by first forming the Lagrangian of the relaxation
    \[\LL(x,\lambda)=x^\top Qx+c^\top x+d+\lambda^\top(Ax-b)\]
    and then, for fixed $\lambda$, minimizing $\LL$ with respect to
    the primal variables $x$.  As~$Q$ is assumed to be positive
    definite, the unique minimizer can be computed from the first
    order optimality condition
    \begin{align}\label{KKT-general}
      \nabla_x \LL(x,\lambda)=2Qx+c+A^\top\lambda=0\Longleftrightarrow x=-\frac{1}{2}Q^{-1}(c+A^\top\lambda).
    \end{align}
    Having $x$ as a function of $\lambda$, we can insert it into the
    Lagrangian $\LL$ yielding the following dual function
    \begin{align*}\LL(\lambda)=\lambda^\top\big(-\frac{1}{4}AQ^{-1}A^\top\big)\lambda-\big(b^\top+\frac{1}{2}c^\top
      Q^{-1}A^\top\big)\lambda-\frac{1}{4}c^\top Q^{-1}c+d.
    \end{align*}
    Defining $\widetilde{Q}:=\frac{1}{4}AQ^{-1}A^\top$,
    $\widetilde{c}:=\frac{1}{2}AQ^{-1}c+b$ and
    $\tilde{d}:=\frac{1}{4}c^\top Q^{-1}c-d$, we can thus write the
    dual of the continuous relaxation of \eqref{QIP} as
    \begin{align}\label{eq:dualprob}
      -\min\ &\lambda^\top\widetilde{Q}\lambda+\widetilde{c}^\top\lambda+\tilde{d}\nonumber\\
      \textnormal{s.t. } &\lambda\geq 0\\
      &\lambda\in\mathbb{R}^m.\nonumber
    \end{align}
    Note that (\ref{eq:dualprob}) is again a convex QP, since
    $\widetilde{Q}$ is positive semidefinite.

    The first crucial difference in considering the dual problem is
    that its dimension changed from $n$ to $m$, which is beneficial if
    $m\ll n$.  The second one is that $\lambda = 0$ is always feasible
    for~\eqref{eq:dualprob}.  Finally, note that having the optimal
    solution $\lambda^{\star}\in\mathbb{R}^m$ of (\ref{eq:dualprob}),
    it is easy to reconstruct the corresponding optimal primal
    solution $x^{\star}\in\mathbb{R}^n$ using the first order
    optimality condition (\ref{KKT-general}).

Within a branch-and-bound framework, a special feature of the dual
    approach is the \emph{early pruning}: we can stop the iteration process and prune the
    node as soon as the current iterate $\lambda_k$ is feasible and
    its objective function value exceeds the current upper bound,
    since each dual feasible solution yields a valid bound. 
    Note however that, in case we cannot prune, an optimal solution of the
    dual problem is required, since it is needed for the computation
    of the corresponding primal solution $x^{\star}$ which in turn is
    needed to decide the enumeration order in the branch-and-bound
    scheme.

 During the tree search it may occur that a node relaxation is infeasible due to the current fixings. 
 In this case infeasibility of the primal problem implies the unboundness of the dual problem. Therefore, during the solution process 
    of the dual problem, an iterate will be reached such that its objective function value exceeds the current upper bound and the node can be pruned. 
    This is why in our implementation of~\texttt{FAST-QPA} we set the following stopping criterion:
   the algorithm stops either if the norm of the projected gradient is less than a given optimality tolerance,
or if an iterate is computed such that its objective function value exceeds the current upper bound.
 More precisely, we declare optimality when the point $\lambda\in \R_+^m$ satisfies the following condition:
  \begin{equation*}
   \|\max\left \{0, -g(\lambda)\right\}\|\le 10^{-5}.
  \end{equation*}
By Propositions~\ref{projgrad0} and~\ref{qxk_inf}, we then have a guarantee that the algorithm stops after a finite number of iterations.

    \subsection{Reoptimization}\label{subsec:reopt}

    At every node of the branch-and-bound tree, we use our algorithm
    \FASTQPA\ described in Section~\ref{sec:feas_as} for solving
    Problem \eqref{eq:dualprob}. A crucial advantage of using an
    active set method is the possibility of working with warmstarts,
    i.e., of passing information on the optimal active set from a
    parent node to its children. In the dual approach the dimension of
    all subproblems is~$m$, independently of the depth~$\ell$ in the
    branch-and-bound tree. When fixing a variable, only the objective
    function changes, given by $\widetilde{Q}$, $\widetilde{c}$
    and~$\widetilde{d}$. So as the starting guess in a
    child node, we choose
    $\mathcal{A}^0:=\mathcal{A}(\lambda^{\star})$, i.e., we use the
    estimated active set for the optimal solution~$\lambda^{\star}$ of the parent
    node, according to Step~3 of Algorithm~\ref{fig:FAST-QPA}. We also pass the solution 
    $\lambda^\star$ to the child nodes to initialize the variables in the line search 
    procedure in Step~4 of Algorithm~\ref{fig:FAST-QPA}, that is we set $\tilde x^0_{\mathcal N^0} = \lambda^\star_{\mathcal N^0}$.  
    Our experimental 
    results presented in Section~\ref{sec:exp}
    show that this warmstarting approach reduces the average number of
    iterations of \FASTQPA\ significantly.

    \subsection{Incremental Computations and Preprocessing}\label{subsec:incremental}

    A remarkable speed-up can be achieved by exploiting the fact that
    the subproblems enumerated in the branch-and-bound tree are
    closely related to each other. Let $\ell\in\{0,\dots,n_1-1\}$ be the
    current depth in the branch-and-bound tree and recall that after
    fixing the first~$\ell$ variables, the problem reduces to the
    minimization of
    \begin{align}\bar f\colon \mathbb{Z}^{n_1-\ell}\times\mathbb{R}^{n-n_1}\rightarrow \mathbb{R},~x\mapsto x^\top Q_\ell x+\bar c^\top x+\bar d\nonumber\end{align}
      over the feasible region $\bar{\mathcal{F}}=\{x\in\mathbb{Z}^{n_1-\ell}\times\mathbb{R}^{n-n_1}\mid
      A_\ell x\leq \bar{b}\}$, where $Q_\ell\succ 0$ is obtained by
      deleting the corresponding $\ell$ rows and columns of $Q$ and $\bar c$
      and~$\bar d$ are adapted properly by
      \[\bar{c}_{j-\ell}:=c_j+2\sum_{i=1}^\ell q_{ij}r_i,\ \textnormal{for }j=\ell+1,\dots,n\]
      and 
      \[\bar{d}:=d+\sum_{i=1}^\ell c_ir_i+\sum_{i=1}^\ell\sum_{j=1}^\ell q_{ij}r_ir_j,\]
      where $r=(r_1,\dots,r_\ell)\in\mathbb{Z}^\ell$ is the current fixing at depth $\ell$.
      %; see~\cite{buchheim2012}. 
      Similarly, $A_\ell$ is obtained by deleting the corresponding~$\ell$ columns
      of $A$ and the reduced right hand side $\bar{b}$ is updated according
      to the current fixing.

      Since we use a predetermined branching order, the reduced matrices $Q_\ell$, $Q_\ell^{-1}$ and~$A_\ell$
      only depend on the depth $\ell$, but not on the specific
      fixings. Along with the reduced matrix $Q_\ell$, the quadratic
      part of the reduced dual objective function $\widetilde{Q}_\ell$
      can then be computed in the preprocessing phase, because they 
      only depend on $Q_\ell$ and $A_\ell$.  The predetermined branching 
      order also allows the computation of the maximum eigenvalues
      $\lambda_{max}(\widetilde{Q}_\ell)$ in the preprocessing phase,
      needed for ensuring proper convergence of our active set method
      as described in Section~\ref{sec:feas_as}; compare~\eqref{eps-prop2}.

      Concerning the linear part $\widetilde{c}$ and the constant part
      $\widetilde{d}$ of the dual reduced problem, both can be
      computed incrementally in linear time per node:
      let~$r=(r_1,\dots,r_\ell)\in\mathbb{Z}^l$ be the current fixing
      at depth $\ell$. By definition of~$\widetilde{c}$, we have
      \[\widetilde{c}(r)=\frac{1}{2}A_\ell Q_\ell^{-1}c(r)+b(r)\;,\]
      where the suffix~$(r)$ always denotes the corresponding data after
      fixing the first~$\ell$ variables to~$r$.
      \begin{theorem}
        After a polynomial time preprocessing, the vector~$\widetilde{c}(r)$ can be constructed
        incrementally in $O(n-\ell+m)$ time per node.
        \label{theorem:1}
      \end{theorem}
      \begin{proof}	
        Defining $y(r):=-\frac{1}{2}Q_\ell^{-1}c(r)$, we have
        \[\frac{1}{2}A_\ell Q_\ell^{-1}c(r)=-A_\ell\cdot y(r).\]
        Note that~$y(r)$ is the unconstrained continuous minimizer
        of~$f(r)$. In~\cite{buchheim2012}, it was shown that $y(r)$ can be
        computed incrementally by \[y(r):=[y(r')+\alpha
          z^{\ell-1}]_{1,\dots,n-\ell}\in\mathbb{R}^{n-\ell}\] for some vector
        $z^{\ell-1}\in\mathbb{R}^{n-\ell+1}$ and $\alpha:=r_{\ell}-y(r')_\ell\in\mathbb{R}$,
        where $r'=(r_1,\dots,r_{\ell-1})$ is the fixing at the
        parent node. This is due to the observation that the continuous minima according 
        to all possible fixings of the next variable lie on a line, for which $z^{\ell-1}$ is the direction. 
        It can be proved that the vectors~$z^{\ell-1}$ only depend on the depth~$\ell$ and
        can be computed in the preprocessing~\cite{buchheim2012}. Updating~$y$
        thus takes~$O(n-\ell)$ time. We now have
        \begin{align*}
          \hspace*{-0.4cm}\widetilde{c}(r)&=-A_\ell[y(r')+\alpha z^{\ell-1}]_{1,\dots,n-\ell}+b(r)\\
          &=-A_\ell[y(r')]_{1,\dots,n-\ell}-\alpha A_\ell[z^{\ell-1}]_{1,\dots,n-\ell}+b(r) \\ 
          &=-(A_{\ell-1}y(r')-y(r')_{n-\ell+1}\cdot A_{\lcdot,n-\ell+1})-\alpha A_\ell[z^{\ell-1}]_{1,\dots,n-\ell}+b(r).	
        \end{align*}
        In the last equation, we used the fact that the first part of the
        computation can be taken over from the parent node by subtracting
        column $n-\ell+1$ of $A$, scaled by the last component of $y(r')$, from
        $A_{\ell-1}y(r')$, which takes~$O(m)$ time. The second part
        $A_\ell[z^{\ell-1}]_{1,\dots,n-\ell}$ can again be computed in the
        preprocessing. The result then follows from the fact that
        also~$b(r)$ can easily be computed incrementally
        from~$b(r')$ in~$O(m)$ time.
      \end{proof}

      \begin{theorem}
        After a polynomial time preprocessing, the scalar $\tilde{d}(r)$ can be
        constructed incrementally in~$O(n-\ell)$ time per node.
        \label{theorem:2}
      \end{theorem}
      \begin{proof}
        Recalling that
        $$\tilde{d}(r)=\frac{1}{4}c(r)^\top Q_\ell^{-1}c(r)-d(r)\;,$$
        this follows from the fact that $y(r)=-\frac 12Q_\ell^{-1}c(r)$
        and~$c(r)$ can be computed in~$O(n-\ell)$ time per node~\cite{buchheim2012}.
      \end{proof}

      \begin{corollary}
        After a polynomial time preprocessing, the dual problem~\eqref{eq:dualprob} can be constructed in~$O(n-\ell+m)$ time per node.
      \end{corollary}

      Besides the effort for solving the QP with the active set method,
      computing the optimal solution of the primal problem from the dual solution is the most time consuming task in each node. 
      The following observation is used to speed up its computation.
      \begin{theorem}
        After a polynomial time preprocessing, the optimal primal
        solution~$x^{\star}(r)$ can be computed from the optimal dual
        solution~$\lambda^{\star}(r)$ in $O(m\cdot (n-\ell))$ time per node.
        \label{theorem:3}
      \end{theorem}
      \begin{proof}
        %% Let $x_\ell^\star\in\mathbb{R}^{n-\ell}$ be the optimal solution of the continuous relaxation of the reduced problem \eqref{prob:reduced} at depth $\ell$ and $\lambda_\ell^\star\in\mathbb{R}^m$ be the optimal solution of the corresponding dual problem of the form \eqref{eq:dualprob}. Then the following observation holds.
        From \eqref{KKT-general} we derive 
        \begin{align*}
          x^{\star}(r)&=-\frac{1}{2}Q_\ell^{-1}\left(\sum_{i=1}^m\lambda^{\star}(r)_ia_i+c(r)\right) \\
          &=y(r)+\sum_{i=1}^m\lambda^{\star}(r)_i\left(-\frac{1}{2}Q_\ell^{-1}a_i\right)\;.\end{align*}
        The first part can again be computed incrementally in $O(n-\ell)$ time per node. For the second part, we observe that $-\frac{1}{2}Q_\ell^{-1}a_i$ can be computed in the preprocessing phase for all $i=1,\dots,m$.
      \end{proof}

      The above results show that the total running time per node is linear
      in~$n-\ell$ when the number~$m$ of constraints is considered a
      constant and when we make the reasonable assumption that Problem~\eqref{quadpr} can be solved in
      constant time in fixed dimension.

      \subsection{Postprocessing}
Using \FASTQPA, we can solve the dual problem of the continuous relaxation in every
node of the branch-and-bound tree, in case it admits a solution, up to high precision.
If the dual problem is bounded, strong duality guarantees the
existence of a primal feasible solution. However, in practice, computing the primal
solution by formula~\eqref{KKT-general}, can affect its feasibility. 
This numerical problem is negligible in the pure integer case 
(where we end up fixing all the variables), while it becomes crucial in the mixed integer case. 
Indeed, when dealing with mixed integer problems, the primal solution of the relaxation in the optimal branch-and-bound node 
(i.e. the node that gives the optimal value) is actually used to build up the solution of the original problem.

%Using our branch-and-bound scheme, we noticed that, even when computing the optimal value
%of Problem~\eqref{QIP} up to an optimality tolerance of $10^{-6}$, the calculation of the 
%corresponding optimal \emph{solution} of Problem~\eqref{QIP} could be problematic in the mixed-integer case.

Hence, in case a high precision is desired for the primal solution,
we propose the following approach: the branch-and-bound node related
to the optimal value gives an optimal fixing
of the integer variables. Therefore, we consider the convex QP that
arises from Problem~\eqref{QIP} under these optimal fixings. Then, we call a 
generic solver to deal with this QP problem (the
values of the primal continuous variables 
obtained by formula \eqref{KKT-general}
can be used as a starting point). 
Since the hardness of Problem~\eqref{QIP} is due to the
integrality constraints on the variables, the running
time for this postprocessing step is negligible.

In the experiments reported below, we apply this approach using
\texttt{CPLEX 12.6} as solver for the QP problem (we choose
$10^{-6}$ as feasibility tolerance). The time required for the
postprocessing step is included in all stated running times.
As noticed above, this postprocessing 
phase is not needed in the pure integer case.

      \section{Experimental Results}\label{sec:exp}

      In order to investigate the potential of our algorithm \texttt{FAST-QPA-BB}, we implemented it
      in \texttt{C++/Fortran~90} and compared it to the MIQP solver of
      \texttt{CPLEX 12.6}. We also tested the branch-and-bound solver
      \texttt{B-BB} of \texttt{Bonmin 1.74}. However, we did not include the
      running times for the latter into the tables since its performance was 
      not competitive at all, not even for mixed-integer instances with a few 
      number of integer variables. All experiments were carried out on Intel Xeon processors
      running at 2.60\;GHz. We used an absolute optimality tolerance
      and a relative feasibility tolerance
      of $10^{-6}$ for all algorithms.

      In order to obtain a feasible solution to Problem~\eqref{QIP}
      and thus an initial upper bound -- or to determine
      infeasibility of~\eqref{QIP} -- we replace the objective
      function in~\eqref{QIP} by the zero function and use the
      \texttt{CPLEX 12.6} ILP solver. In principle, the algorithm also
      works when setting the initial upper bound to a very large
      value. Then it is either replaced as soon as a feasible solution
      is found in some branch-and-bound node, or it will remain
      unchanged until the algorithm terminates, in which case
      Problem~\eqref{QIP} must be infeasible.

      Altogether, we randomly
      generated 1600 different problem instances for~\eqref{QIP},
      considering percentages of integer variables
      $p:=\frac{n_1}{n}\in\{0.25, 0.50, 0.75, 1.0\}$. For $p=0.25$
      (0.5 / 0.75 / 1.0), we chose $n\in\{50,100,150,200,250\}$
      $(\{50,75,100,125,150\}$ / $\{50,60,70,80,90\}$ /
      $\{50,55,60,65,70\})$, respectively. The number of constraints
      $m$ was chosen in $\{1,10,25,50\}$. For each combination of $p$,
      $n$ and $m$, we generated 10 instances. For every group of
      instances with a given percentage of integer variables~$p$, the
      parameter $n$ was chosen up to a number such that at least one
      of the tested algorithms was not able to solve all of the 10
      instances to optimality for $m=1$ within our time limit of 3~cpu
      hours.

      For generating the positive definite matrix~$Q$, we chose $n$
      eigenvalues $\lambda_i$ uniformly at random from $[0,1]$ and
      orthonormalized $n$ random vectors $v_i$, each entry of which
      was chosen uniformly at random from $[-1,1]$, then we set
      $Q=\sum_{i=1}^n\lambda_iv_iv_i^{\top}$.
      The entries of~$c$ were chosen uniformly at random from
      $[-1,1]$, moreover we set~$d=0$. For the entries of~$A$ and~$b$,
      we considered two different choices:
      \begin{itemize}
      \item[(a)] the entries of~$b$ and~$A$ were chosen uniformly at random from $[-1,1]$,
      \item[(b)] the entries of~$A$ were chosen uniformly at random
        from $[0,1]$ and we set $b_i=\frac 12{\sum_{j=1}^na_{ij}}$, $i=1,\dots,m$.
      \end{itemize}
      The constraints of type~(b) are commonly used to create
      hard instances for the knapsack problem.
At \url{www.mathematik.tu-dortmund.de/lsv/instances/MIQP.tar.gz} all instances are publicly available.

      The performance of the considered algorithms for instances of type~(a) can be
      found in Tables~\ref{tab:100percent}--\ref{tab:25percent}. 
      We do not inlcude the tables for 
      instances of type~(b), since there are no significant
      differences in the results, except that they are in general 
      easier to solve for our algorithm as well as for \texttt{CPLEX}.
      All running
      times are measured in cpu seconds. The tables include the following
      data for the comparison between \texttt{FAST-QPA-BB} and \texttt{CPLEX
        12.6}: numbers of instances solved within the time limit, average preprocessing time, 
      average running times, average number of branch-and-bound nodes, average number of iterations
      of \FASTQPA\ in the root node and average number of iterations
      of \FASTQPA\ per node in the rest of the enumeration tree. All averages are taken over
      the set of instances solved within the time limit.

      \begin{table}
        {\small
          \centering

          \begin{tabular}{|c|c||c|c|c|c|c|c||c|c|c|} 
            \hline
            \multicolumn{2}{|c||}{inst} & \multicolumn{6}{|c||}{\texttt{FAST-QPA-BB}} & \multicolumn{3}{|c|}{\texttt{CPLEX 12.6}} \\
            $n$ & $m$ & \# & ptime & time & nodes  & it root & it & \# & time & nodes \\
            \hline
            \hline
            50 &   1 & 10 & 0.00 &   6.83 & 9.24e+6 &       1.50 &     1.16  & 10 &    53.50 & 5.06e+5 \\
            50 &  10 & 10 & 0.00 &  83.15 & 3.16e+7 &      3.80 &     1.54 & 10 &   189.50 & 1.45e+6 \\ 
            50 &  25 &  9 & 0.01 & 2337.87 & 1.62e+8 &      8.56 &     2.09 & 10 & 2413.50 & 1.25e+7 \\ 
            50 &  50 & 0 & - & - & - & - & - & 0 & - & - \\ \hline
            55 &   1 & 10 & 0.00 &   28.47 & 3.62e+7 &       1.60 &     1.25 & 10 &   210.71 & 1.75e+6 \\
            55 &  10 & 10 & 0.00 &  427.66 & 1.46e+8 &      3.40 &     1.48 & 10 &  1154.91 & 7.53e+6 \\ 
            55 &  25 &  4 & 0.01 & 3843.30 & 3.37e+8 &      7.50 &     1.86 & 5 &  3949.40 & 1.82e+7 \\
            55 &  50 & 0 & - & - & - & - & - & 0 & - & - \\ \hline
            60 &   1 & 10 & 0.00 &  133.32 & 1.60e+8 &       1.30 &     1.11 & 9 &   353.89 & 2.61e+6 \\
            60 &  10 &  8 & 0.01 & 1894.81 & 6.86e+8 &       2.62 &     1.51 &  7 &  3477.91 & 2.04e+7 \\ 
            60 &  25 &  2 & 0.02 & 8963.19 & 7.55e+8 &      5.50 &     2.00 & 1 &  6149.31 & 2.68e+7 \\
            60 &  50 & 0 & - & - & - & - & - & 0 & - & - \\ \hline
            65 &   1 & 10 & 0.01 &  349.11 & 4.13e+8 &       1.40 &     1.15 & 10 &  2105.36 & 1.31e+7 \\
            65 &  10 &  8 & 0.01 &  4010.42 & 1.50e+9 &      2.75 &     1.48 &  4 &  5503.84 & 2.83e+7 \\ 
            65 &  25 & 0 & - & - & - & - &  - & 0 & - & - \\ 
            65 &  50 & 0 & - & - & - & - &  - & 0 & - & - \\ \hline
            70 &   1 & 10 & 0.01 &  1113.47 & 1.30e+9 &       1.60 &     1.27 & 7 &  5133.59 & 2.88e+7  \\
            70 &  10 &  4 & 0.01 &  6915.67 & 2.38e+9 &       2.50 &     1.51 & 0 & - & - \\ 
            70 &  25 & 0 & - & - & - & - &  - & 0 & - & - \\ 
            70 &  50 & 0 & - & - & - & - &  - & 0 & - & - \\
            \hline
          \end{tabular}\vspace*{0.1cm}
          \caption{Results for instances of type~(a) with $p=1.0$.}
          \label{tab:100percent}

          \vspace{0.8cm}
          \centering
          \begin{tabular}{|c|c||c|c|c|c|c|c||c|c|c|} 
            \hline
            \multicolumn{2}{|c||}{inst} & \multicolumn{6}{|c||}{\texttt{FAST-QPA-BB}} & \multicolumn{3}{|c|}{\texttt{CPLEX 12.6}} \\
            $n$ & $m$ & \# & ptime & time & nodes & it root & it & \# & time & nodes \\
            \hline
            \hline            

60	&	1	&	10	&	0.00	&	1.81	&	2.01e+6	&	1.30	&	1.11	&	10	&	12.12	&	1.12e+5		\\
60	&	10	&	10	&	0.01	&	5.65	&	1.82e+6	&	2.80	&	1.44	&	10	&	17.63	&	1.34e+5		\\
60	&	25	&	10	&	0.02	&	32.60	&	2.36e+6	&	9.10	&	1.65	&	10	&	26.49	&	1.51e+5		\\
60	&	50	&	10	&	0.08	&	641.91	&	6.64e+6	&	55.80	&	2.14	&	10	&	129.51	&	4.80e+5		\\ \hline
70	&	1	&	10	&	0.01	&	12.15	&	1.30e+7	&	1.60	&	1.20	&	10	&	84.02	&	6.67e+5		\\
70	&	10	&	10	&	0.01	&	27.11	&	8.19e+6	&	3.40	&	1.43	&	10	&	77.70	&	5.07e+5		\\
70	&	25	&	10	&	0.03	&	159.81	&	1.23e+7	&	9.20	&	1.60	&	10	&	183.43	&	8.85e+5		\\
70	&	50	&	10	&	0.08	&	2222.15	&	2.79e+7	&	18.60	&	1.96	&	10	&	593.70	&	1.89e+6		\\ \hline
80	&	1	&	10	&	0.02	&	65.98	&	6.51e+7	&	1.40	&	1.12	&	10	&	446.60	&	2.91e+6		\\
80	&	10	&	10	&	0.03	&	151.77	&	4.37e+7	&	3.80	&	1.42	&	10	&	386.33	&	2.08e+6		\\
80	&	25	&	10	&	0.04	&	963.38	&	7.06e+7	&	10.30	&	1.60	&	10	&	791.62	&	3.20e+6		\\
80	&	50	&	7	&	0.09	&	3339.39	&	5.38e+7	&	15.29	&	1.82	&	9	&	1903.79	&	5.27e+6		\\ \hline
90	&	1	&	10	&	0.03	&	417.90	&	3.79e+8	&	1.30	&	1.11	&	10	&	2332.52	&	1.25e+7		\\
90	&	10	&	10	&	0.04	&	1538.99	&	4.36e+8	&	3.00	&	1.40	&	10	&	2745.36	&	1.26e+7		\\
90	&	25	&	10	&	0.06	&	4468.73	&	3.55e+8	&	5.90	&	1.55	&	9	&	3094.58	&	1.08e+7		\\
90	&	50	&	1	&	0.09	&	5361.11	&	1.02e+8	&	10.00	&	1.72	&	4	&	5199.56	&	1.27e+7		\\ \hline
100	&	1	&	6	&	0.05	&	1897.13	&	1.69e+9	&	1.33	&	1.10	&	5	&	5918.25	&	2.61e+7		\\
100	&	10	&	6	&	0.06	&	5893.96	&	1.52e+9	&	4.83	&	1.39	&	0	&	-	&	-		\\
100	&	25	&	1	&	0.06	&	4897.75	&	3.95e+8	&	5.00	&	1.51	&	0	&	-	&	-		\\
100	&	50	&	1	&	0.11	&	3622.67	&	8.05e+7	&	15.00	&	1.74	&	1	&	2664.13	&	5.47e+6		\\ \hline

          \end{tabular}\vspace*{0.1cm}
          \caption{Results for instances of type~(a) with $p=0.75$.}
          \label{tab:75percent}
          %\end{minipage}
        }
        %\end{sidewaystable}
      \end{table}
 
      \begin{table}
        {\small
          \centering
          \begin{tabular}{|c|c||c|c|c|c|c|c||c|c|c|} 
            \hline
            \multicolumn{2}{|c||}{inst} & \multicolumn{6}{|c||}{\texttt{FAST-QPA-BB}} & \multicolumn{3}{|c|}{\texttt{CPLEX 12.6}} \\
            $n$ & $m$ & \# & ptime & time & nodes &  it root & it & \# & time & nodes \\
            \hline
            \hline
  
50	&	1	&	10	&	0.00	&	0.02	&	9.83e+3	&	1.50	&	1.16	&	10	&	0.18	&	1.26e+3		\\
50	&	10	&	10	&	0.00	&	0.03	&	5.46e+3	&	3.80	&	1.48	&	10	&	0.16	&	7.47e+2		\\
50	&	25	&	10	&	0.01	&	0.11	&	4.68e+3	&	8.20	&	1.71	&	10	&	0.23	&	8.70e+2		\\
50	&	50	&	10	&	0.06	&	1.14	&	8.30e+3	&	21.40	&	2.35	&	10	&	0.61	&	1.63e+3		\\ \hline
75	&	1	&	10	&	0.01	&	0.25	&	1.76e+5	&	1.60	&	1.21	&	10	&	2.20	&	1.30e+4		\\
75	&	10	&	10	&	0.02	&	0.68	&	1.65e+5	&	3.60	&	1.42	&	10	&	2.50	&	1.17e+4		\\
75	&	25	&	10	&	0.02	&	2.23	&	1.68e+5	&	9.80	&	1.57	&	10	&	3.80	&	1.47e+4		\\
75	&	50	&	10	&	0.09	&	12.08	&	1.70e+5	&	17.00	&	1.88	&	10	&	6.56	&	1.74e+4		\\ \hline
100	&	1	&	10	&	0.05	&	13.84	&	1.08e+7	&	1.50	&	1.15	&	10	&	76.62	&	3.20e+5		\\
100	&	10	&	10	&	0.05	&	14.13	&	3.46e+6	&	4.80	&	1.39	&	10	&	63.18	&	2.35e+5		\\
100	&	25	&	10	&	0.07	&	56.86	&	4.35e+6	&	8.50	&	1.53	&	10	&	87.39	&	2.55e+5		\\
100	&	50	&	10	&	0.13	&	322.71	&	5.39e+6	&	51.80	&	1.72	&	10	&	216.89	&	4.38e+5		\\ \hline
125	&	1	&	10	&	0.11	&	193.38	&	1.33e+8	&	1.40	&	1.13	&	10	&	2449.15	&	6.70e+6		\\
125	&	10	&	10	&	0.13	&	516.50	&	1.15e+8	&	3.40	&	1.37	&	10	&	2265.19	&	5.70e+6		\\
125	&	25	&	10	&	0.14	&	1342.59	&	1.03e+8	&	8.80	&	1.46	&	10	&	2111.49	&	4.22e+6		\\
125	&	50	&	10	&	0.21	&	4524.8	&	9.28e+7	&	52.60	&	1.62	&	10	&	2843.59	&	4.22e+6		\\ \hline
150	&	1	&	9	&	0.22	&	4011.63	&	2.34e+9	&	1.78	&	1.25	&	2	&	5026.81	&	9.86e+6		\\
150	&	10	&	6	&	0.24	&	6857.64	&	1.29e+9	&	3.50	&	1.38	&	0	&	-	&	-		\\
150	&	25	&	3	&	0.24	&	9440.33	&	6.65e+8	&	7.33	&	1.46	&	0	&	-	&	-		\\
150	&	50	&	0	&	     -    	&	      -  	&	    -     	&	    -  	&	        - 	&	0	&	-	&	-		\\ \hline

          \end{tabular}\vspace*{0.1cm}
          \caption{Results for instances of type~(a) with $p=0.5$.}
          \label{tab:50percent}
          \vspace{0.8cm}
          \centering
                    \begin{tabular}{|c|c||c|c|c|c|c|c||c|c|c|} 
            \hline
            \multicolumn{2}{|c||}{inst} & \multicolumn{6}{|c||}{\texttt{FAST-QPA-BB}} & \multicolumn{3}{|c|}{\texttt{CPLEX 12.6}} \\
            $n$ & $m$ & \# & ptime & time & nodes &  it root & it & \# & time & nodes \\
            \hline
            \hline

50	&	1	&	10	&	0.01	&	0.01	&	1.31e+2	&	1.50	&	1.18	&	10	&	0.01	&	3.44e+1		\\
50	&	10	&	10	&	0.00	&	0.01	&	1.96e+2	&	3.80	&	1.53	&	10	&	0.02	&	4.46e+1		\\
50	&	25	&	10	&	0.01	&	0.03	&	1.40e+2	&	8.20	&	1.98	&	10	&	0.03	&	3.85e+1		\\
50	&	50	&	10	&	0.06	&	0.10	&	1.14e+2	&	21.40	&	3.85	&	10	&	0.07	&	3.32e+1		\\ \hline
100	&	1	&	10	&	0.06	&	0.09	&	1.05e+4	&	1.50	&	1.16	&	10	&	0.50	&	1.06e+3		\\
100	&	10	&	10	&	0.06	&	0.11	&	3.88e+3	&	4.80	&	1.40	&	10	&	0.45	&	6.32e+2		\\
100	&	25	&	10	&	0.06	&	0.18	&	4.91e+3	&	8.50	&	1.56	&	10	&	0.55	&	7.36e+2		\\
100	&	50	&	10	&	0.12	&	0.75	&	8.24e+3	&	51.80	&	1.80	&	10	&	1.30	&	1.53e+3		\\ \hline
150	&	1	&	10	&	0.25	&	0.62	&	1.68e+5	&	1.70	&	1.23	&	10	&	7.49	&	1.31e+4		\\
150	&	10	&	10	&	0.22	&	0.97	&	1.24e+5	&	3.00	&	1.39	&	10	&	7.37	&	1.09e+4		\\
150	&	25	&	10	&	0.24	&	2.85	&	1.56e+5	&	6.40	&	1.46	&	10	&	12.81	&	1.65e+4		\\
150	&	50	&	10	&	0.31	&	4.74	&	7.77e+4	&	11.20	&	1.60	&	10	&	9.50	&	8.69e+3		\\ \hline
200	&	1	&	10	&	0.58	&	14.79	&	6.38e+6	&	1.50	&	1.16	&	10	&	283.67	&	2.73e+5		\\
200	&	10	&	10	&	0.52	&	21.32	&	3.00e+6	&	3.20	&	1.37	&	10	&	244.09	&	2.19e+5		\\
200	&	25	&	10	&	0.57	&	74.93	&	4.03e+6	&	5.90	&	1.44	&	10	&	344.88	&	2.56e+5		\\
200	&	50	&	10	&	0.68	&	400.99	&	7.44e+6	&	40.60	&	1.54	&	10	&	818.13	&	4.83e+5		\\ \hline
250	&	1	&	10	&	1.12	&	329.54	&	1.21e+8	&	1.50	&	1.15	&	7	&	4413.64	&	2.54e+6		\\
250	&	10	&	10	&	1.10	&	617.59	&	7.35e+7	&	3.80	&	1.35	&	9	&	5291.97	&	2.67e+6		\\
250	&	25	&	10	&	1.17	&	2115.57	&	1.01e+8	&	6.90	&	1.41	&	5	&	4984.95	&	2.08e+6		\\
250	&	50	&	9	&	1.28	&	3709.07	&	6.27e+7	&	49.11	&	1.51	&	4	&	5444.02	&	1.90e+6		\\ \hline

          \end{tabular}\vspace*{0.1cm}
          \caption{Results for instances of type~(a) with $p=0.25$.}
          \label{tab:25percent}
        }
      \end{table}

        \begin{table}
        {\small
          \centering
          \begin{tabular}{|c|c||c|c|c|c||c|c|c|c|} 
            \hline
            \multicolumn{2}{|c||}{inst} & \multicolumn{4}{|c||}{\texttt{FAST-QPA-BB}} & \multicolumn{4}{|c|}{\texttt{FAST-QPA-BB-NP}} \\
            $n$ & $m$ & \# & time & nodes & it & \# & time & nodes & it\\
            \hline
            \hline                  
            50 &   1 & 10 & 6.83 & 9.24e+6 &    1.16 & 10 &     8.79 & 9.24e+6 & 1.48\\  
            50 &  10 & 10 & 83.15 & 3.16e+7 &    1.54 & 10 &    145.95 & 3.16e+7 & 2.74\\ 
            50 &  25 &  9 & 2337.87 & 1.62e+8 &   2.09 & 7 &    4587.82    & 1.35e+8 & 4.25\\
            50 &  50 & 0  & -       &    -    &   -    & 0 &    - &  & -\\ \hline 
            55 &   1 & 10 & 28.47 & 3.62e+7 &     1.25 & 10 &   34.86 & 3.62e+7 & 1.74\\
            55 &  10 & 10 & 427.66 & 1.46e+8 &     1.48 & 10 &  743.12 & 1.46e+8 & 2.53\\ 
            55 &  25 &  4 & 3843.30 & 3.37e+8 &     1.86 & 3 &  4216.79 & 1.56e+8 & 3.63\\
            55 &  50 & 0 & - & - & - & 0 & - & - & -\\ \hline
            60 &   1 & 10 & 133.32 & 1.60e+8 &     1.11 & 10 &   154.74 & 1.60e+8 & 1.34\\
            60 &  10 &  8 & 1894.81 & 6.86e+8 &     1.51 &  8 &  3259.91 & 6.86e+8 & 2.62 \\ 
            60 &  25 &  2 & 8963.19 & 7.55e+8 &     2.00 & 0 &  - &  & -\\
            60 &  50 & 0 & - & - & - & 0 & - & - & -\\ \hline
            65 &   1 & 10 & 349.11 & 4.13e+8 &     1.15 & 10 &  410.14 & 4.13e+8 & 1.47\\
            65 &  10 &  8 & 4010.42 & 1.50e+9 &    1.48 &  7 &  5238.44 & 1.23e+9 & 2.52\\ 
            65 &  25 & 0 & - & - &  - & 0 & - & - & -\\ 
            65 &  50 & 0 & - & - &  - & 0 & - & - & -\\ \hline
            70 &   1 & 10 & 1113.47 & 1.30e+9 &    1.27 & 10 &  1318.95 & 1.30e+9  & 1.81\\
            70 &  10 &  4 & 6915.67 & 2.38e+9 &    1.51 & 1 & 8714.63 & 1.31e+9 & 2.95\\ 
            70 &  25 & 0 & - & - &  - & 0 & - & - & -\\ 
            70 &  50 & 0 & - & - &  - & 0 & - & - & -\\
            \hline
          \end{tabular}\vspace*{0.1cm}
          \caption{Results for instances of type~(a) with $p=1.0$ turning early pruning on/off.}
          \label{tab:early_pruning}
        }
      \end{table}

      From our experimental results, we can conclude that when $p\in\{0.25;0.50\}$ 
      \texttt{FAST-QPA-BB} clearly
      outperforms \texttt{CPLEX 12.6} for $m$ up to 25 and is at 
      least competitive to \texttt{CPLEX 12.6} if $p=0.75$. 
      For the mixed-integer case we can see that
      the average running times of \texttt{FAST-QPA-BB} compared to
      \texttt{CPLEX 12.6} are the better, the bigger the percentage of
      continuous variables is, even with a larger number of
      constraints. For the pure integer case we still outperform
      \texttt{CPLEX 12.6} for $m$ up to 10. For all instances, the preprocessing 
      time is negligible.
      
      This experimental study shows that \texttt{FAST-QPA-BB} is able
      to solve 644 instances of type (a) to optimality, 
      while \texttt{CPLEX 12.6} can only solve 598 instances. 
      Note that the average number of
      branch-and-bound nodes in our dual approach is approximately 30
      times greater than that needed by \texttt{CPLEX
        12.6}. Nevertheless the overall running times of our approach
      are much faster for moderate sizes of~$m$, emphasizing both the
      quick enumeration process within the branch-and-bound tree and
      the benefit of using reoptimization. Note that the performance
      of our approach highly depends on $m$. As the number of
      constraints grows, the computational effort for both solving the
      dual problem and recomputing the primal solution (see
      Theorem~\ref{theorem:3}), is growing as well.

      In Table~\ref{tab:early_pruning} we compare the performance of \texttt{FAST-QPA-BB} 
       with \texttt{FAST-QPA-BB-NP}, a version in which the early pruning is not implemented
      (see Section~\ref{subsec:dual}).
      We show the results for the pure integer instances of type~(a) with $p=1.0$.
      The benefits from the early pruning are evident: the average number of iterations 
      of \texttt{FAST-QPA} is decreased leading to faster running times so that $9$ more instances 
      can be solved.

      Our experimental results also underline the strong performance
      of \FASTQPA. The number of iterations of \FASTQPA\ needed in the
      root node of our branch-and-bound algorithm is very small on
      average: for~$m=50$ it is always below~$60$ and often much
      smaller. Using warmstarts, the average number of iterations
      drops to 1--6. 

      Besides the tables of average running times, we visualized our
      results by performance profiles in Figure~\ref{fig:perf_plots},
      % and \ref{fig:perf_plotsAb}
      as proposed in~\cite{dolan2002}. They confirm the result that \texttt{FAST-QPA-BB} outperforms
      \texttt{CPLEX 12.6} significantly.
      %We denote our set of solvers by $S$, our set of instances by $I$, 
      %consisting of the $n_p$ randomly generated instances, and choose as 
      %performance measure the computational time to solve any instance to optimality. 
      %For any solver $s\in S$ and instance $i\in I$ let $t_{i,s}$ 
      %be the running time of algorithm $s$ to solve instance $i$. 
      %If any algorithm $s$ is not able to find the optimal solution of an instance $i$ within the given time-limit, we set $t_{i,s}$ to a value $M\geq\max\{t_{p,s}\mid p\in P,s\in S\}$. 
      %The performance ratio for a given solver $s$ and instance $i$ is definred by
      %\[r_{i,s}:=\frac{t_{i,s}}{\min\{t_{i,s}\mid s\in S\}}\;.\]
      %The cumulative distribution function for the performance ratio is then defined by
      %\[\rho_s(\tau):=\frac{1}{n_p}\left|i\in I\ : \ r_{i,s}\leq \tau\right|\;.\]
      %It measures the probability for some $s\in S$ that the performance ratio is within a factor of $\tau\in\mathbb{R}$ of the best possible ratio.
      %The performance profiles are given in Figure~\ref{fig:perf_plots} and \ref{fig:perf_plotsAb}. \texttt{FAST-QPA-BB} outperforms \texttt{CPLEX 12.6}.
      \begin{figure}
        \centering
        \psfrag{fastqpabb}[l][l]{{\tiny \texttt{FAST-QPA-BB}}}
        \psfrag{cplex}[l][l]{{\tiny \texttt{CPLEX 12.6}}}
        \includegraphics[scale=0.2,trim=2cm 0 1cm 0, clip]{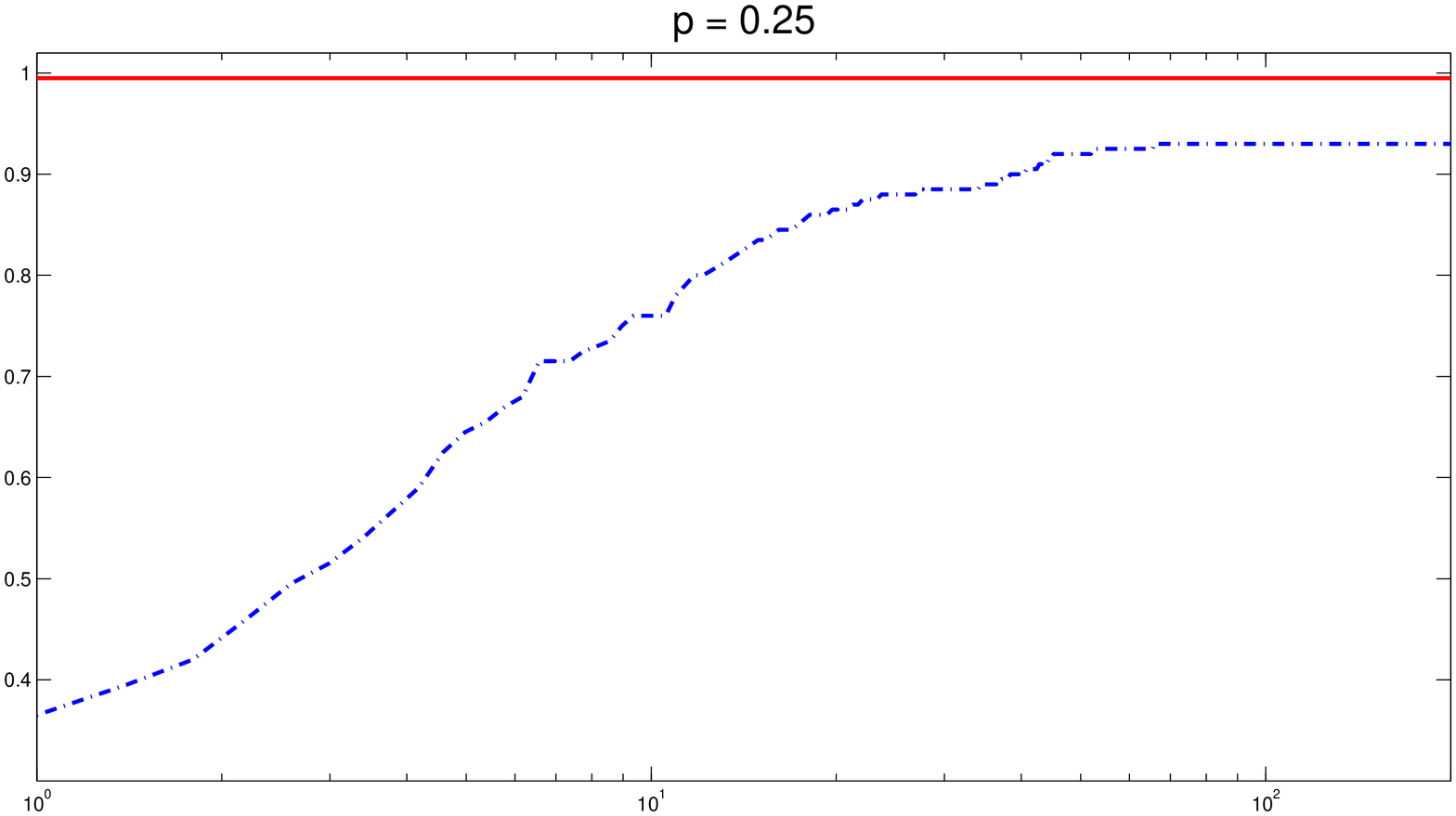}
        \includegraphics[scale=0.2,trim=2cm 0 1cm 0, clip]{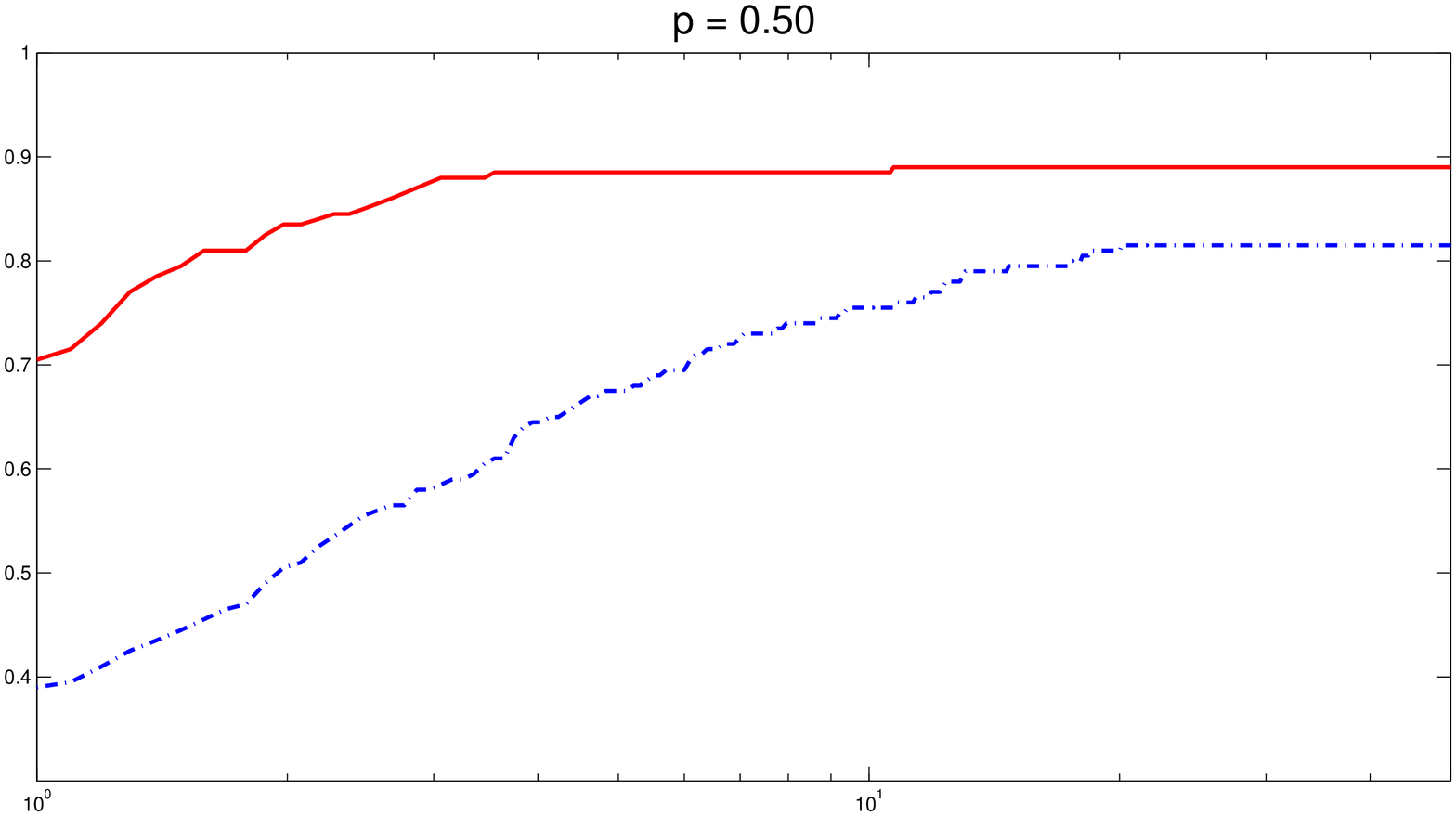}\\
        \includegraphics[scale=0.2,trim=2cm 0 1cm 0, clip]{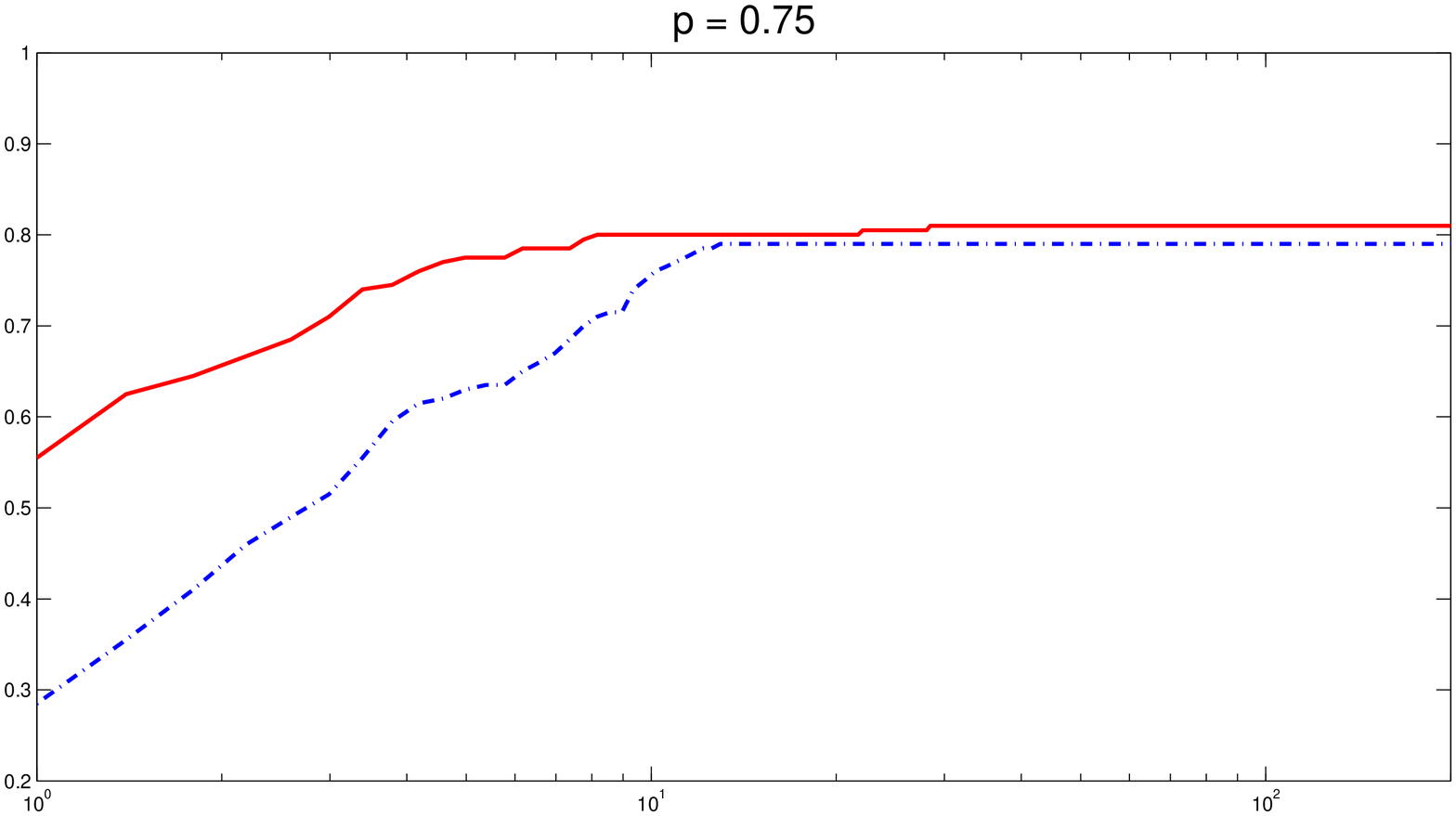}
        \includegraphics[scale=0.2,trim=2cm 0 1cm 0, clip]{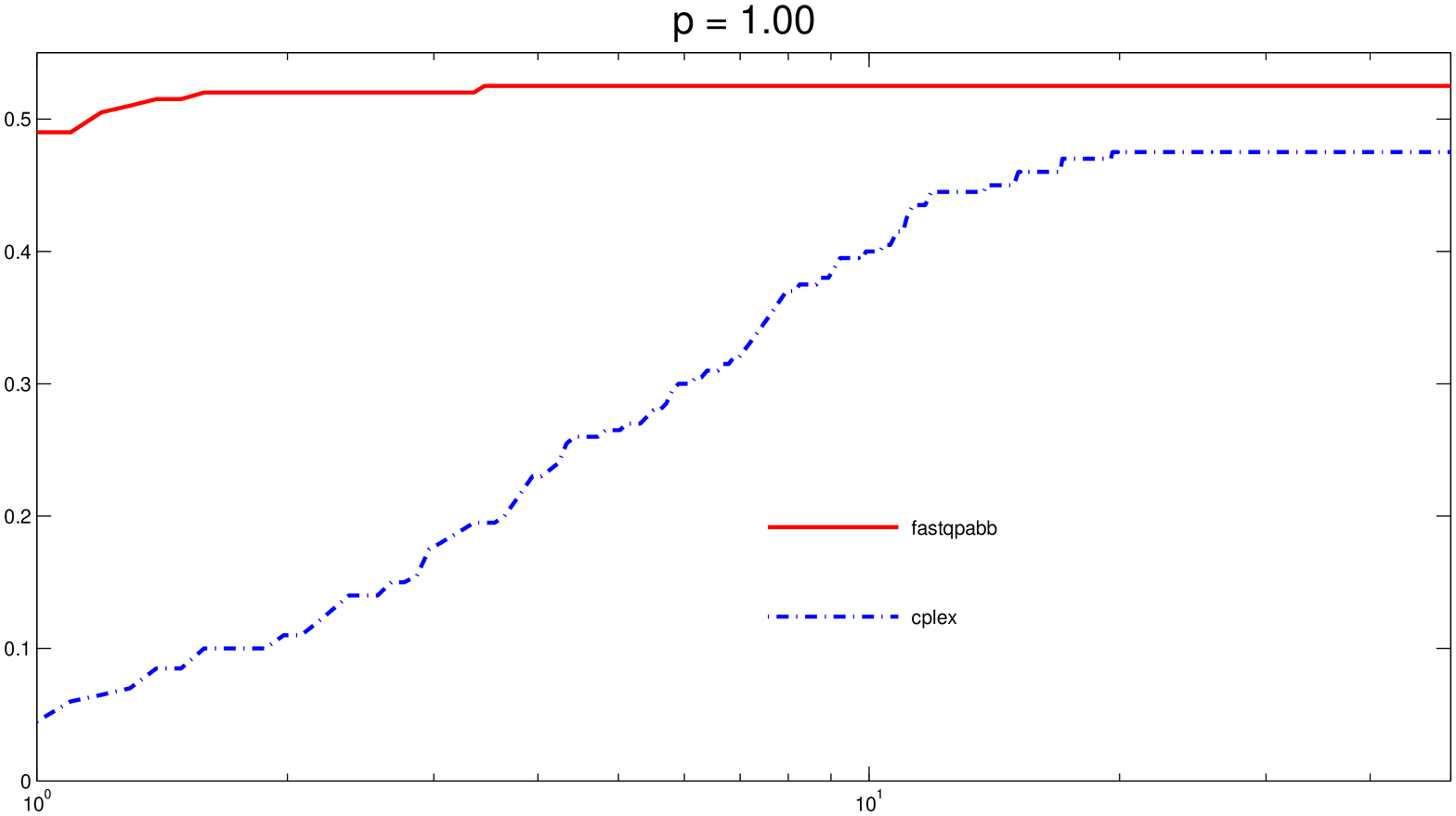}
        \caption{Performance profiles for all instances of type~(a)}
        \label{fig:perf_plots}
      \end{figure}

      % \begin{figure}
      %   \centering
      %   \psfrag{fastqpabb}[l][l]{{\tiny \texttt{FAST-QPA-BB}}}
      %   \psfrag{cplex}[l][l]{{\tiny \texttt{CPLEX 12.6}}}
      %   %\includegraphics[width=\textwidth,trim=1cm 0 0 0, clip]{figures/PP.pdf}
      %   \includegraphics[scale=0.2,trim=2cm 0 1cm 0, clip]{figures/PP25new.eps}
      %   \includegraphics[scale=0.2,trim=2cm 0 1cm 0, clip]{figures/PP50new.eps}\\
      %   \includegraphics[scale=0.2,trim=2cm 0 1cm 0, clip]{figures/PP75new.eps}
      %   \includegraphics[scale=0.2,trim=2cm 0 1cm 0, clip]{figures/PP100new.eps}
      %   \caption{Performance profiles for all instances of type~(b)}
      %   \label{fig:perf_plotsAb}
      % \end{figure}

      \section{Conclusions}\label{sec:summary}
      We presented a new branch-and-bound algorithm for convex
      quadratic mixed integer minimization problems based on the use
      of an active set method for computing lower bounds. Using a dual
      instead of a primal algorithm considerably improves the running
      times, as it may allow an early pruning of the node. Moreover,
      the dual problem only contains non-negativity constraints,
      making the problem accessible to our tailored active set method
      \FASTQPA. Our sophisticated rule to estimate the active set
      leads to a small number of iterations of \FASTQPA\ in the root node
      that however grows as the number of constraints increases. This
      shows that for a large number of constraints the QPs addressed
      by \FASTQPA\ are nontrivial and their solution time has a big impact on the total
      running time, since we enumerate a large number of nodes in the
      branch-and-bound tree. Nevertheless, reoptimization helps to
      reduce the number of iterations of \FASTQPA\ per node
      substantially, leading to an algorithm that outperforms \texttt{CPLEX 12.6}
      on nearly all problem instances considered.

      %\newpage
      \bibliographystyle{siam}
      \bibliography{references}

\end{document}